\documentclass{amsart}
\usepackage{pgf,amsthm,amscd,amssymb,verbatim,epsf,amsmath,eurosym,amsfonts,mathrsfs,graphicx,tikz-cd,lscape}
\usepackage{amsthm,amscd,amssymb,verbatim,epsf,amsmath,amsfonts,mathrsfs,graphicx,dirtytalk,pdfpages,mdframed,tikz-cd,xy}
\usepackage[colorlinks=true,linkcolor=blue,citecolor=blue]{hyperref}
\usepackage[font=small,labelfont=bf]{caption} 
\usepackage{epigraph,wrapfig}
\begin{document}
\theoremstyle{plain}
\newtheorem{Thm}{Theorem}
\newtheorem{Cor}[Thm]{Cor}
\newtheorem{Ex}[Thm]{Example}
\newtheorem{Con}[Thm]{Conjecture}
\newtheorem{Main}{Main Theorem}
\newtheorem{Lem}[Thm]{Lemma}
\newtheorem{Prop}[Thm]{Proposition}

\theoremstyle{definition}
\newtheorem{Def}[Thm]{Definition}
\newtheorem{Note}[Thm]{Note}
\newtheorem{Question}[Thm]{Question}

\theoremstyle{remark}
\newtheorem{notation}[Thm]{Notation}
\renewcommand{\thenotation}{}

\errorcontextlines=0
\renewcommand{\rm}{\normalshape}%
\newcommand{\transv}{\mathrel{\text{\tpitchfork}}}
\makeatletter
\newcommand{\tpitchfork}{%
  \vbox{
    \baselineskip\z@skip
    \lineskip-.52ex
    \lineskiplimit\maxdimen
    \m@th
    \ialign{##\crcr\hidewidth\smash{$-$}\hidewidth\crcr$\pitchfork$\crcr}
  }%
}
\makeatother

\title[The Toponogov Conjecture]{Proof of the Toponogov Conjecture \\ on Complete Surfaces}
\author{Brendan Guilfoyle}\address{Brendan Guilfoyle\\
          School of STEM\\
          Munster Technological University\\
          Tralee \\
          Co. Kerry \\
          Ireland.}
\email{brendan.guilfoyle@mtu.ie}
\author{Wilhelm Klingenberg}
\address{Wilhelm Klingenberg\\
 Department of Mathematical Sciences\\
 University of Durham\\
 Durham DH1 3LE\\
 United Kingdom}
\email{wilhelm.klingenberg@durham.ac.uk }

\begin{abstract}
We prove a conjecture of Toponogov on complete convex planes, namely that such planes must contain an umbilic point, albeit at infinity. Our proof is indirect. It uses Fredholm regularity of an associated Riemann-Hilbert boundary value problem and an existence result for holomorphic discs with Lagrangian boundary conditions, both of which apply to a putative counterexample. 

Corollaries of the main theorem include a Hawking-Penrose singularity-type theorem, as well as the proof of a conjecture of Milnor's from 1965 in the convex case. 

\end{abstract}

\maketitle

\section{Introduction and Results}\label{s:intro}

In this paper we prove the following Theorem:
\vspace{0.1in}
\begin{Thm}\label{t:1}
Let  $P \subset{{\mathbb R}^3}$ be the image of a $C^{3,\alpha}$--smooth embedding of \; ${\mathbb R}^2$ into ${\mathbb R}^3$. Then the following three properties cannot hold simultaneously:
\begin{itemize}
\item[I.] $P$ is geodesically complete with respect to the induced metric,
\item[II.] $P$ is convex: $\kappa_1 \cdot \kappa_2  \geq   0  $,
\item[III.] $P$ is uniformly non-umbilic: $\exists C>0$ such that $\inf_P |\kappa_1 - \kappa_2|> C$.
\end{itemize}
Here $\kappa_1,\kappa_2$ are the principal curvatures of $P$.\
\end{Thm}
\vspace{0.1in}

The method of proof uses the existence of holomorphic discs to contradict Fredholm regularity of an associated Riemann-Hilbert problem with negative index, thereby proving the non-existence of boundary planes satisfying I, II and III for which the problem {\em would} be regular. This method closely follows the authors' approach to the Carath\'eodory Conjecture \cite{G20} \cite{Gk02} \cite{Gk01} \cite{GK23}.

The first step takes any sufficiently smooth plane $P$ in ${\mathbb R}^3$ to the collection of its oriented normal lines, considered as a plane ${\mathcal P}$ in the space ${\mathbb L}$ of all oriented lines. Elementary properties of this correspondence imply that ${\mathcal P}$ is Lagrangian with respect to the canonical symplectic structure $\Omega$ on ${\mathbb L}$. When $P$ is strictly convex,  ${\mathcal P}$ is the graph of a section of the bundle ${\mathbb L}=TS^2\rightarrow {\mathbb S}^2$ taking an oriented line to its direction. Moreover, a point on $P$ is umbilic iff the corresponding oriented line considered as a point in ${\mathcal P}$ is a complex point with respect to the canonical complex structure ${\mathbb J}$ on ${\mathbb L}$ \cite{gak4}.

Thus take the non-umbilic convex plane $P$ and construct the totally real, Lagrangian plane ${\mathcal P}$ in the neutral K\"ahler 4-manifold $({\mathbb L},{\mathbb J},\Omega,{\mathbb G})$. Properties I and II imply that this Lagrangian plane is a section over an open hemisphere, possibly with some boundary points. This follows directly from Lemma's 1 to 4 of \cite{Top} and is discussed further in the next section.

The second step considers ${\mathcal P}$ as the boundary condition in the boundary value problem for the Cauchy-Riemann equation. The following will be called the Riemann-Hilbert problem (RHP) with boundary condition  ${\mathcal P}$: the existence or otherwise of properly embedded discs $f:(D,\partial D)\rightarrow ({\mathbb L},{\mathcal P})$ which are holomorphic with respect to ${\mathbb J}$:
\[
\bar{\partial}f(z):=df(z) \circ j - {\mathbb J}(f) \circ df(z) =  0  \qquad\forall   z \in D 
\]
\[
f(z) \in {\mathcal P}   \qquad \forall  z \in \partial D,
\]
where $j$ is the standard complex structure on the unit disc $D \subset {\mathbb C}$.

In Section \ref{s:rhp} the following is proven:

\vspace{0.1in}
\begin{Thm}\label{t:2}
Assume that $P$ satisfies Properties I and III and that $f$ solves the Riemann-Hilbert problem for boundary ${\mathcal P}$.  Then for a dense open set ${\mathcal W}$ of Lagrangian perturbations of ${\mathcal P}$, the associated Riemann-Hilbert problem with boundary condition $\tilde{\mathcal P}\in{\mathcal W}$ is Fredholm regular. As a result, the analytic index at  $\tilde{\mathcal P}\in{\mathcal W}$  is 
\[
I( \tilde{\mathcal P})={\mbox{ dim ker }}\bar{\partial}=\mu({\mathbb L},T\tilde{\mathcal P})+2,
\]
where $\mu$ is the Keller-Maslov index of the loop $f(\partial D) \hookrightarrow \tilde{\mathcal P} \hookrightarrow ({\mathbb L},\Omega) $ .
\end{Thm}
\vspace{0.1in}

While the complex and symplectic structures are compatible, the associated inner product ${\mathbb G}$ has neutral signature $(2,2)$ and therefore surfaces may be both Lagrangian and complex. In general this poses difficulties for the standard Riemann-Hilbert problem.

However, the complex points on ${\mathcal P}$ correspond to umbilic points on $P$, where $\kappa_1=\kappa_2$, and so Property III implies that ${\mathcal P}$ is uniformly totally real. Theorem \ref{t:2} then follows from standard arguments as outlined in Section \ref{s:rhp} - see also \cite{Gk01}.

The next step is to prove the existence of solutions of the Riemann-Hilbert problem under exactly these circumstances:

\vspace{0.1in}

\begin{Thm}\label{t:mcf}
Let $P$ be a $C^{3+\alpha}$-smooth plane in ${\mathbb R}^3$ satisfying Properties I, II and III. Let ${\mathcal P}\subset {\mathbb L}$ be the oriented normals of $P$. 

Then $\exists f:D\rightarrow {\mathbb L}$ with $f\in C^{1+\alpha}_{loc}(D)\cap C^0(\overline{D})$ satisfying
\begin{enumerate}
\item[(i)] $f$ is holomorphic,
\item[(ii)] $f(\partial D)\subset{\mathcal P}$.
\end{enumerate}
\end{Thm} 

\vspace{0.1in}

The construction of the holomorphic disc is carried out by mean curvature flow. In particular, we consider the following initial boundary value problem:

\vspace{0.1in}

\begin{center}\fbox{\parbox{4.8in}{
\begin{center}{\Large{ \bf I.B.V.P.}}\end{center}
{\it
Consider a family of positive sections $f_s:D\rightarrow TS^2$ such that
\[
\frac{d f}{ds}=H,
\]
with initial and boundary conditions:
\vspace{0.1in}
\begin{enumerate}
\item[(i)] $f_0(D)=D_0,$
\item[(ii)]$f_s(\partial D)\subset \tilde{\mathcal P}$,
\item[(iii)] the hyperbolic angle $B$ between $Tf_s(D)$ and $T\tilde{\mathcal P}$ is constant along $f_s(\partial D)$,
\item[(iv)] $f_s(\partial D)$ is asymptotically holomorphic: $|\bar{\partial}f_s|=C/(1+s)$,
\end{enumerate} 
\vspace{0.1in}
where $H$ is the mean curvature vector of $f_s(D)$, and $D_0$ and $\tilde{\mathcal P}$ are given positive sections.
\vspace{0.1in}
}
}
}
\end{center}
\vspace{0.1in}

Here positive means spacelike: the induced metric is positive definite. In what follows we refer to $f_s(\partial D)$ as the {\it edge} of the flowing surface, which lies on the boundary plane $\tilde{\mathcal P}$. The boundary plane  $\tilde{\mathcal P}$ is obtained from the Lagrangian plane ${\mathcal P}$ of normals to $P$ by adding a holomorphic twist (see Section \ref{s:bdryinit} for full details). 

The idea of using mean curvature flow to construct holomorphic discs has been attempted before - see for example \cite{chenli}. Such flows have been considered in definite K\"ahler manifolds where mean curvature flow almost inevitably develops singularities. In contrast, a recent and growing body of literature has demonstrated that mean curvature flow of spacelike submanifolds in {\em indefinite} manifolds, even with higher codimension, can be significantly better behaved \cite{EaH} \cite{huang} \cite{LaS} \cite{smock}. 

In our case, we use the general interior gradient estimate for higher codimensional mean curvature flow in indefinite spaces established in \cite{Gk02}. The {\bf I.B.V.P.} is a quasilinear parabolic system and short time existence for the flow is established in Theorem \ref{t:ste}. The proof consists of checking the Lopatinski-Shapiro conditions for the boundary conditions and then using standard Schauder theory. 

Long time existence under certain conditions is then proven in Theorem \ref{t:lte}. In particular, having added a sufficiently large holomorphic twist to the boundary plane to make it positive at the origin, we can choose an initial surface and hyperbolic angle so that the flow exists for all time.

Long-time existence is ensured by uniform positivity of the flowing disc. Uniform positivity in the interior of the flowing disc is established by showing that the conditions required in the compact case, Theorem 1 of \cite{Gk02}, namely the timelike convergence condition and containment in a compact region, hold for this flow. We then establish uniform positivity at the edge by careful consideration of the boundary conditions and finding a priori bounds on the 2-jet of the flowing surface. Convergence to a holomorphic disc is established in Section \ref{ss:hol_disc}.

As in the standard setting, a holomorphic disc is maximal with respect to the neutral metric. The fact that a holomorphic disc, which satisfies a first order equation, can be found by a second order flow, which would usually be expected to converge to a maximal disc, is rather special. In particular, the asymptotic holomorphic boundary condition is a critical element. A simpler version of this phenomenon can be seen for the stationary rotationally symmetric case in the following Corollary of Theorem 3 in \cite{gak8}:

\vspace{0.1in}
\begin{Cor}
    A definite maximal rotationally symmetric surface in $TS^2$ with holomorphic boundary is holomorphic.
\end{Cor}
\vspace{0.1in}

These mean curvature flow results are established in the following sections. The neutral geometry of $TS^2$ is summarized in Section \ref{s:nkg}, while mean curvature flow with boundary in $TS^2$ is considered in detail in Section \ref{s:mcf}. Short-time and long-time existence for the flow, along with the  existence of holomorphic discs with Lagrangian boundary conditions, are established in the Section \ref{s:ibvp}. 

\vspace{0.1in}

Theorems \ref{t:2} and \ref{t:mcf} imply Theorem \ref{t:1} as follows. Since the totally real boundary ${\mathcal P}$ is simply connected, the Keller-Maslov index is zero. Thus after subtracting off the dimension of the M\"obius group, the dimension of the space of holomorphic discs is $-1$. Thanks to Theorem \ref{t:2}, by a perturbation of the boundary ${\mathcal P}$ any existing  holomorphic disc would disappear. But this contradicts Theorem \ref{t:mcf}. We conclude that there does not exist a plane $P$ satisfying Properties I, II and III.

A remarkable aspect of Theorem \ref{t:1} is its three natural corollaries. According to the Theorem, a plane that satisfies Properties I and II cannot satisfy Property III, and we have proven a conjecture of Victor Andreevich Toponogov \cite{Top}:

\vspace{0.1in}
\begin{Cor}
Every $C^{3,\alpha}$-smooth complete convex embedding of the plane $P$, satisfies $\inf_P |\kappa_1-\kappa_2|=0$.
\end{Cor}
\vspace{0.1in}

The second corollary is the following singularity theorem that follows from the fact that a plane that satisfies Properties II and III cannot satisfy Property I. 

\vspace{0.1in}
\begin{Cor}
A convex uniformly non-umbilic plane is geodesically incomplete i.e. it has an edge in ${\mathbb R}^3$.
\end{Cor}
\vspace{0.1in}

The final corollary comes from Properties I and III excluding Property II, which establishes the convex case of a conjecture attributed to John Milnor from 1965 \cite{Mil} - see \cite{Xav} for a recent discussion.
\vspace{0.1in}
\begin{Cor}
A complete uniformly umbilic-free plane cannot be convex. 
\end{Cor}
\vspace{0.1in}

In the next section we summarize the work of Toponogov on complete convex planes. Section \ref{s:rhp} contains the proof of Theorem \ref{t:2} that the associated Riemann-Hilbert problem is Fredholm regular. A summary of the required neutral K\"ahler geometry appears in Section \ref{s:nkg}. Sections \ref{s:mcf} and \ref{s:ibvp} contains the proof of Theorem \ref{t:mcf} on the existence of holomorphic discs, while the final section contains the proof of Theorem \ref{t:1}.

\vspace{0.1in}

\section{Complete convex planes after Toponogov}\label{s:top}

In this section we review the work of Toponogov who proved the following:
\vspace{0.1in}
\begin{Thm} \cite{Top} \label{t:Top}
Let the embedded plane $P\subset{\mathbb R}^3$ satisfy conditions I, II and III. Then $P$ is a graph over a bounded convex domain, such that the Gauss image is an open hemisphere, possibly with some points on the boundary.
\end{Thm}
\vspace{0.1in}
\begin{proof}
In Lemmas 1 to 4 of \cite{Top}, it is established that the plane $P$ is a graph over a bounded convex domain. 

In addition, after a rotation and translation, it is shown that $P$ can be represented in cylindrical coordinates $(r,\varphi,z)$, related to the standard Euclidean coordinates $(x,y,z)$ through
\[
x=r\cos\varphi \qquad y=r\sin\varphi \qquad z=z,
\]
so that the plane $P$ is
\[
r=r(z,\varphi)\qquad z\geq0 \qquad 0\leq\varphi\leq2\pi.
\]
Moreover, it is established that
\[
\lim_{z\rightarrow\infty}r(z,\varphi)=r_0(\varphi)<\infty,\qquad\lim_{z\rightarrow\infty}r_z(z,\varphi)=0,
\]
\[
r_z(z,\varphi)\geq0, \qquad r_{zz}(z,\varphi)\leq0\qquad{\mbox{for }}z>0.
\]
In the strictly convex case, we have that $r_z(z,\varphi)>0$ and the Gauss image is an open hemisphere. For weakly convex planes, the Gauss image may include points on the boundary of the hemisphere.
\end{proof}
\vspace{0.1in}

\begin{center}
    \includegraphics[width=0.5\textwidth]{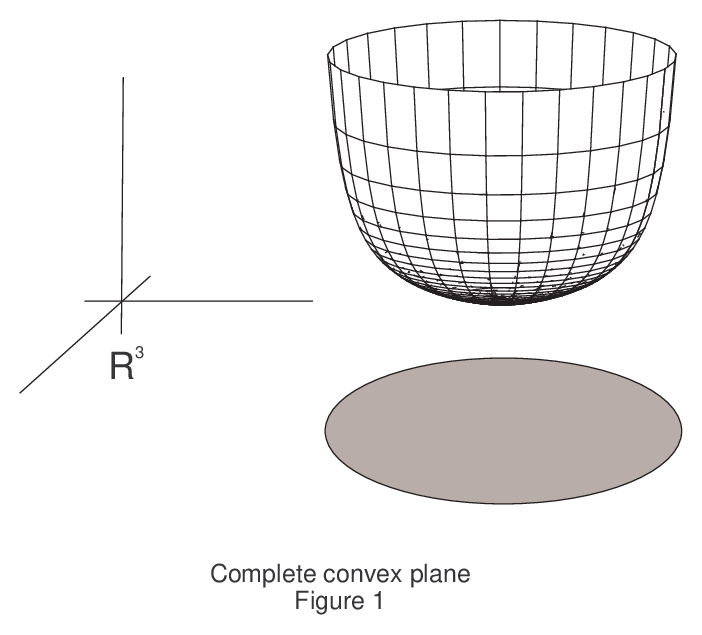}
\end{center}

In Figure 1 a complete convex plane is illustrated. It is rotationally symmetric and therefore has an umbilic point at its tip, thus satisfying properties I and II, but not property III.

\vspace{0.1in}

\section{Fredholm Regularity of the Associated Riemann-Hilbert Problem}\label{s:rhp}

Given a smoothly embedded convex plane $P\subset{\mathbb R}^3$ one constructs the associated Riemann-Hilbert boundary problem as described earlier. The oriented normal lines to such a plane $P$ form a plane in the set of all oriented lines in ${\mathbb R}^3$ and we consider the problem of finding a holomorphic disc whose boundary lies on this plane.

In more detail, let ${\mathbb L}$ be the space of oriented lines in ${\mathbb R}^3$ with canonical complex structure ${\mathbb J}$ and symplectic structure $\Omega$. These form a K\"ahler structure in which the metric ${\mathbb G}$ has neutral signature. It is easily seen that ${\mathbb L} \cong T{\mathbb S}^2$, with canonical projection $\pi:{\mathbb L}\rightarrow {\mathbb S}^2$ taking an oriented line to its unit direction vector in  ${\mathbb S}^2 \subset {\mathbb R}^3 $.

Let $P\subset{\mathbb R}^3$ be a smooth convex plane with a chosen orientation, and consider the set of oriented normal lines to $P$. These form a plane ${\mathcal P}\hookrightarrow {\mathbb L}$ with the following properties:

\vspace{0.1in}

\begin{Prop}
If the plane $P$ satisfies Property III, then ${\mathcal P}$ is a uniformly totally real Lagrangian plane. In addition, the induced metric on ${\mathcal P}$ is Lorentz.
\end{Prop}
\begin{proof}
By Proposition 10 in \cite{gak4} the oriented normals to any surface form a Lagrangian surface in ${\mathbb L}$. Property III implies that ${\mathcal P}$ is uniformly totally real, as umbilic points on $P$ correspond to complex points on ${\mathcal P}$. Moreover, the signature of the induced metric is either Lorentz or degenerate, the latter occurring at complex points. Since the plane is totally real, the induced metric is Lorentz. 

\end{proof}
\vspace{0.1in}

Given any smooth proper plane $P$, the {\it associated Riemann-Hilbert Problem} seeks to find a properly embedded holomorphic disc $f:(D,\partial D)\rightarrow({\mathbb L},{\mathcal P})$. Moreover we seek to describe the set of all such holomorphic discs. 

To analyze holomorphic discs in the complex surface $({\mathbb L},{\mathbb J})$ with boundary lying on ${\mathcal P}$, set the following notation. Fix  $\alpha\in(0,1)$, $s\geq 1$ and denote by $C^{k+\alpha}$ and $H^{1+s}$ the usual H\"older and Sobolev spaces, respectively. Denote
\[
H^{1+s}({\mathbb L}):=\{f:D\rightarrow{\mathbb L}\;|\; f\in H^{1+s}\}.
\]
Define the space of {\it H\"older Lagrangian boundary conditions} by
\[
{\mathcal L}{\mbox{ag}} :=  \left\{ \;   {\mathcal P}\subset {\mathbb L} \; {\mbox{is }}\; C^{2+\alpha} \;{\mbox{-smooth proper Lagrangian plane}}\right\}.
\] 
This is a topological space defined by $C^{2+\alpha}$ proximity in the ambient space.

Assume there exists a $C^{3+\alpha}$-smooth proper plane $P_0\subset{\mathbb R}^3$ that is uniformly non-umbilic, and let ${\mathcal P}_0$ be the corresponding $C^{2+\alpha}$-smooth uniformly totally real Lagrangian plane in ${\mathbb L}$. 
\vspace{0.1in}
\begin{Prop}\label{p:banach}
There exists a neighbourhood ${\mathcal U}\subset{\mathcal L}{\mbox{ag}} $ of ${\mathcal P}_0$ which is a Banach manifold modelled on the Banach space $C^{2+\alpha}({\mathbb R}^2)$.
\end{Prop}
\begin{proof}
Since ${\mathcal P}_0$ is uniformly totally real, the complex structure ${\mathbb J}$ takes tangent vectors on ${\mathcal P}_0$ to vectors transverse to ${\mathcal P}_0$. Thus ${\mathcal L}{\mbox{ag}}$ is locally modeled on the Banach space of functions on ${\mathcal P}_0$: 
\[
T_{{\mathcal P}_0}{\mathcal L}{\mbox{ag}}=\left\{\;{\mathbb J}({\mbox{grad}}_{\mathbb G}\Phi)\;\;\left|\;\; \Phi\in C^{2+\alpha}({\mathcal P}_0)\;\right.\right\},
\] 
where ${\mathbb G}$ is the ambient K\"ahler metric and ${\mbox{grad}}_{\mathbb G}$ is the gradient with respect to the induced non-degenerate Lorentz metric (non-degenerate because ${\mathcal P}_0$ is totally real).
\end{proof}

For $s > 1, (2 - \alpha)/2 = 1/s $ and a relative class $A\in\pi_2({\mathbb L},{\mathcal P})$, the space of {\it parameterized Sobolev-regular discs with Lagrangian boundary condition} is defined by
\[
{\mathcal{F}}_A\equiv \left\{\;(f, {\mathcal P})\in H^{1+s}({\mathbb L})\times{\mathcal U}\;\;\left|\;\; [f]=A,\; f(\partial D)\subset {\mathcal P}  \;\right.\right\},
\]
where ${\mathcal U}$ is the Banach manifold neighbourhood of ${\mathcal P}_0$ as above. The space ${\mathcal{F}}_A$ is a Banach manifold and so the projection $\pi:{\mathcal{F}}_A\rightarrow {\mathcal{U}}$: $\pi(f, {\mathcal P})= {\mathcal P}$ is a Banach bundle. 

For $(f, {\mathcal P})\in{\mathcal F}_A$ define $\bar{\partial}f={\textstyle{\frac{1}{2}}}(df\circ j- {\mathbb J} \circ df)$, where
$j$ is the complex structure on $D$. Then $\bar{\partial}f\in H^s(f^*T ^{01}{\mathbb L}) \equiv H^s(f^*T{\mathbb L}) $ and we define the space of sections
\[
H^s\equiv\bigcup_{(f, {\mathcal P})\in{\mathcal F}_A}H^s(f^*T{\mathbb L}).
\]
This is a Banach vector bundle over ${\mathcal F}_A$ and the operator $\bar{\partial}$ is a section of this bundle.

\vspace{0.1in}
\begin{Def}
The {\it set of holomorphic discs with Lagrangian boundary condition} is defined by
\[
{\mathcal M}_A\equiv\left\{\;(f, {\mathcal P})\in{\mathcal F}_A\;\;\left|\;\; \bar{\partial}f=0 \;\right.\right\}.
\]
\end{Def}
\vspace{0.1in}

\vspace{0.1in}
\begin{Prop}\label{p:banach2}
There exists a neighbourhood of ${\mathcal P}_0$, denoted ${\mathcal V}\subset{\mathcal L}\mbox{ag}$, such that ${\mathcal M}_A\cap\pi^{-1}({\mathcal V})$ is a Banach submanifold of ${\mathcal F}_A\cap\pi^{-1}({\mathcal V})$.
\end{Prop}

\vspace{0.1in}

Consider the linearization of $\bar{\partial}$ at $(f, {\mathcal P}) \in \mathcal{F}_A$ with respect to any ${\mathbb J}$-parallel connection on $H^{1+s}({\mathcal F}_A)$:
\[
\nabla_{(f, {\mathcal P})} \bar{\partial}: H^{1+s}(f^*T{\mathbb L} \otimes f^*T {\mathcal P}) \to H^s(f^*T{\mathbb L}).
\]
The following key points about this operator are standard:
\vspace{0.1in}

\begin{Prop}\cite{oh}\label{p:oh}
$\nabla_{(f,{\mathcal P})}\bar{\partial}$ is Fredholm and therefore has finite dimensional kernel and cokernel. The analytic index of this operator $I={\mbox{dim ker}}(\nabla_{(f, {\mathcal P})} \bar{\partial}) -{\mbox{dim
coker}}(\nabla_{(f, {\mathcal P})} \bar{\partial})$ is related to the Keller-Maslov index $\mu$ of the edge by
\[
I=\mu({\mathbb L},T{\mathcal P})+2.
\]
\end{Prop}
                    
\vspace{0.1in} 

\begin{Prop} \cite{oh}\label{p:fredholm}
There exists a dense open set ${\mathcal W}\subset{\mathcal V}$ containing ${\mathcal P}_0$ such that any (not multiply covered) holomorphic disc with edge in ${\mathcal P}\in{\mathcal W}$ is Fredholm-regular i.e. ${\mbox{dim coker}}(\nabla_{(f, {\mathcal P})} \bar{\partial})=0$.
\end{Prop}
\begin{proof}
The proof follows from ellipticity of $\bar{\partial}$, the fact that the projection map $\pi$ is Fredholm and the Sard-Smale theorem. 
\end{proof}
\vspace{0.1in}
This proves Theorem \ref{t:2}.

\vspace{0.1in}


\section{ Neutral K\"ahler Geometry of {$TS^2$}{}}\label{s:nkg}

In this section we give details of the neutral K\"ahler geometry of the space of oriented lines, paying particular attention to those parts that play a role later.

\subsection{The neutral K\"ahler surface and submanifolds}\label{s:nkm}

\begin{Def}
A {\it neutral K\"ahler surface} is a 4-manifold ${\mathbb M}$ endowed with a complex structure ${\mathbb J}$, a symplectic structure $\Omega$ and a metric ${\mathbb G}$ of signature $(++--)$ which are compatible in the sense that
\[
{\mathbb G}({\mathbb J}\cdot,{\mathbb J}\cdot)={\mathbb G}(\cdot,\cdot)
\qquad\qquad
{\mathbb G}({\mathbb J}\cdot,\cdot)=\Omega(\cdot,\cdot).
\]
\end{Def}
For such a structure we have the following identity:

\vspace{0.1in}

\begin{Prop}\label{p:callib}\cite{gak9}
Let $({\mathbb{M}},{\mathbb{J}},\Omega,{\mathbb{G}})$  be a neutral K\"ahler surface and let $p\in{\mathbb{M}}$ and $v_1,v_2\in T_p{\mathbb{M}}$ span a plane. Then
\[
\Omega(v_1,v_2)^2-det[v_1,v_2,{\mathbb J}v_1,{\mathbb J}v_2]={\mbox{det }}{\mathbb{G}}(v_i,v_j).
\]

\end{Prop}
\vspace{0.1in}

We turn now to the special case of $TS^2$. In order to compute geometric quantities, we introduce local coordinates. These are readily supplied by lifting the standard complex coordinate $\xi$ (obtained by stereographic projection from the south pole on $S^2$) to complex coordinates ($\xi,\eta$) on $TS^2$. In particular, identify $(\xi,\eta)\in{\mathbb {C}}^2$ with the vector
\[
\eta\frac{\partial}{\partial \xi}+\bar{\eta}\frac{\partial}{\partial \bar{\xi}}\in T_\xi S^2.
\]
These coordinates are holomorphic with respect to the complex structure ${\mathbb J}$:
\[
{\mathbb J}\left(\frac{\partial}{\partial \xi}\right)=i\frac{\partial}{\partial \xi}
\qquad\qquad
{\mathbb J}\left(\frac{\partial}{\partial \eta}\right)=i\frac{\partial}{\partial \eta},
\]
and the symplectic 2-form and neutral metric have the following local expressions:
\[
\Omega=4(1+\xi\bar{\xi})^{-2}{\mathbb{R}}\mbox{e}\left(d\eta\wedge d\bar{\xi}-\frac{2\bar{\xi}\eta}{1+\xi\bar{\xi}}d\xi\wedge d\bar{\xi}\right),
\]
\begin{equation}\label{e:metric}
{\mathbb{G}}=4(1+\xi\bar{\xi})^{-2}{\mathbb{I}}\mbox{m}\left(d\bar{\eta} d\xi+\frac{2\bar{\xi}\eta}{1+\xi\bar{\xi}}d\xi d\bar{\xi}\right).
\end{equation}

\begin{Def}
The canonical coordinates $(\xi,\bar{\xi})$ are called {\it Gauss coordinates} and $R=|\xi|$ is the {\it Gauss radius}. 
\end{Def}

This metric on $TS^2$ is of neutral signature $(++--)$ and is conformally and scalar flat, but not K\"ahler-Einstein. It sits within a larger class of natural scalar flat neutral metrics on $TN$, where $(N,g)$ is a Riemannian 2-manifold \cite{gak4}.   

We turn now to immersed surfaces in $TS^2$. Such 2-parameter families of oriented lines are often referred to as line congruences. Consider surfaces which are graphs of local sections of the bundle $\pi:TS^2\rightarrow S^2$. Such local sections are given by $\xi\mapsto (\xi,\eta=F(\xi,\bar{\xi}))$, for some function $F:U\rightarrow{\mathbb C}$, for $U\subset{\mathbb C}$.

\vspace{0.1in}
\begin{Def}\label{d:spinco}
For a section $\eta=F(\xi,\bar{\xi})$, introduce the weighted complex slopes of $F$:
\[
\sigma=-\partial \bar{F} \qquad\qquad
\rho=\vartheta+i\lambda=(1+\xi\bar{\xi})^2\partial [F(1+\xi\bar{\xi})^{-2}].
\]
Here, and throughout, $\partial$ represents differentiation with respect to $\xi$. The functions $\lambda$ and $\sigma$ are commonly referred to as the {\it twist} and {\it shear} of the underlying family $P$ of oriented lines in ${\mathbb R}^3$ \cite{PaR}.
\end{Def}

For later use we also introduce the notation
\[
\Delta=\lambda^2-|\sigma|^2 \qquad\qquad \mu=\frac{|\sigma|}{|\lambda|}.
\]
\vspace{0.1in}

Note the two identities, which follow from these definitions:
\begin{equation}\label{e:id1}
-(1+\xi\bar{\xi})^2\partial\left[\frac{\bar{\sigma}}{(1+\xi\bar{\xi})^2}\right]=\bar{\partial}\rho+\frac{2F}{(1+\xi\bar{\xi})^2},
\end{equation}
\begin{equation}\label{e:id2}
{\mathbb I}{\mbox{m}}\;\partial\left\{(1+\xi\bar{\xi})^2\partial\left[\frac{\bar{\sigma}}{(1+\xi\bar{\xi})^2}\right]\right\}
    =\partial\bar{\partial}\lambda+\frac{2\lambda}{(1+\xi\bar{\xi})^2}.
\end{equation}

The geometric significance of $\lambda$ and $\sigma$ are:

\vspace{0.1in}

\begin{Prop}\cite{gak2}
A plane ${\mathcal P}$ given by a local section $\eta=F(\xi,\bar{\xi})$ is Lagrangian iff $\lambda=0$ and is holomorphic iff $\sigma=0$.
\end{Prop}

\vspace{0.1in}

For the normal congruence ${\mathcal P}$ of a plane $P$ in ${\mathbb R}^3$ the following relationship between the quantities $\sigma$ and $\rho$ and the principal curvatures and directions of ${P}$ hold:

\begin{Prop}
Let $P$ be a convex plane in ${\mathbb R}^3$  and ${\mathcal P}\subset TS^2$ be the surface formed by the 2-parameter family of oriented normals to $P$. Then ${\mathcal P}$ is the graph of a section and $\lambda=0$.

Moreover
\[
|\sigma|={\textstyle{\frac{1}{2}}}|\kappa_1^{-1}-\kappa_2^{-1}|
\qquad
r+\rho={\textstyle{\frac{1}{2}}}(\kappa_1^{-1}+\kappa_2^{-1}),
\]
where $\kappa_1,\kappa_2$ are the principal curvatures of ${P}$ and $r$ is the support function of $P$ (see below). The argument of $\sigma$ gives the principal directions of $P$.
\end{Prop}
\vspace{0.1in}

To construct the 1-parameter family of parallel surfaces in ${\mathbb R}^3$ from a Lagrangian section consider the following.

\vspace{0.1in}

\begin{Prop}\cite{gak2}
Let ${\mathcal P}$ be a local Lagrangian section given by $\xi\mapsto (\xi,\eta=F(\xi,\bar{\xi}))$, for some function $F:{\mathbb C}\rightarrow{\mathbb C}$. Then there exists a real function $\xi\mapsto r(\xi,\bar{\xi})$, satisfying
\[
\bar{\partial}r=\frac{2F}{(1+\xi\bar{\xi})^2}.
\]
Such a function is defined up to an additive real constant $C$. In terms of Euclidean coordinates the surfaces 
\begin{equation}\label{e:mt}
x^1+ix^2=\frac{2(F-\bar{F}\xi^2)+2\xi(1+\xi\bar{\xi})r}{(1+\xi\bar{\xi})^2},
\qquad
x^3=\frac{-2(F\bar{\xi}+\bar{F}\xi)+(1-\xi^2\bar{\xi}^2)r}{(1+\xi\bar{\xi})^2},
\end{equation}
are orthogonal to the oriented lines of $P$.
\end{Prop}

\vspace{0.1in}

\begin{Def}
The function $r:P\rightarrow{\mathbb R}$ in the previous Proposition is called the {\it support function}. Given a point $p$ on a convex plane $P$, $r$ is the distance between $p$ and the point closest to the origin on the normal line through $p$.

The surfaces obtained by replacing $r$ by $r+C$ are called {\it parallel surfaces}.
\end{Def}

\vspace{0.1in}

For the induced metric we have:

\vspace{0.1in}

\begin{Prop}\label{p:ts2indmet}
The metric induced on the graph of a section by the K\"ahler metric is given in coordinates ($\xi,\bar{\xi}$) by;
\[
g=\frac{2}{(1+\xi\bar{\xi})^2}\left[\begin{matrix}
i\sigma & -\lambda\\
-\lambda & -i\bar{\sigma}\\
\end{matrix}
\right],
\]
with inverse
\[
g^{-1}=\frac{(1+\xi\bar{\xi})^2}{2(\lambda^2-\sigma\bar{\sigma})}\left[\begin{matrix}
i\bar{\sigma} & -\lambda\\
-\lambda & -i\sigma\\
\end{matrix}
\right].
\]
\end{Prop}
\begin{proof}
This follows from pulling back the neutral metric (\ref{e:metric}) along a local section $\eta=F(\xi,\bar{\xi})$.
\end{proof}

\vspace{0.1in}

\begin{Def}
A plane ${\mathcal P}\subset TS^2$ is {\it positive} if the induced metric on ${\mathcal P}$ is positive definite.
\end{Def}

\vspace{0.1in}

\begin{Prop}\label{p:indmet}
The induced metric on a Lagrangian surface is Lorentz, except at complex points, where it is degenerate. The induced metric on a holomorphic surface is positive, except at complex points, where it is degenerate. 
\end{Prop}
\begin{proof}
By the previous Proposition we see that the determinant of the induced metric is $4(1+\xi\bar{\xi})^{-4}(\lambda^2-\sigma\bar{\sigma})$, and the result follows. This is, in fact, a special case of Proposition \ref{p:callib}.
\end{proof}

\vspace{0.1in}

\begin{Def}\label{d:perpdist}
Let $\chi:TS^2\rightarrow {\mathbb R}$ be the map that takes an oriented line to the square of the perpendicular distance from the line to the origin. 
\end{Def}
\vspace{0.1in}

\begin{Prop}
The coordinate expression for $\chi$ is:
\[
\chi^2(\xi,\bar{\xi},\eta,\bar{\eta})=\frac{4\eta\bar{\eta}}{(1+\xi\bar{\xi})^2}.
\]
\end{Prop}
\begin{proof}
The point on an oriented line $(\xi,\eta)$ which lies closest to the origin has Euclidean coordinates (see equation (\ref{e:mt}))
\[
x^1_0+ix^2_0=-\frac{2(\eta-\bar{\eta}\xi^2)}{(1+\xi\bar{\xi})^2},
\qquad
x^3_0=-\frac{2(\eta\bar{\xi}+\bar{\eta}\xi)}{(1+\xi\bar{\xi})^2},
\]
and so the perpendicular distance to the origin is
\[
\chi^2= (x^1_0)^2+(x^2_0)^2+(x^3_0)^2 =\frac{4\eta\bar{\eta}}{(1+\xi\bar{\xi})^2}.
\] 
\end{proof}

\vspace{0.1in}

In order to continue, we introduce geometric tools which will prove useful later.

\vspace{0.1in}

\subsection{The second fundamental form of a positive surface}\label{s:nkgndff}

Let ${\mathcal P}\rightarrow TS^2$ be an immersed plane and assume that the induced metric on ${\mathcal P}$ is positive, so that for $\gamma\in {\mathcal P}$ we have the orthogonal splitting $T_\gamma TS^2=T_\gamma {\mathcal P}\oplus N_\gamma {\mathcal P}$. In what follows we omit the subscript $\gamma$.

\vspace{0.1in}

\begin{Prop}\label{p:frames}
If ${\mathcal P}$ is a positive plane given by the graph $\xi\rightarrow(\xi,\eta=F(\xi,\bar{\xi}))$, then the following vector fields form an orthonormal basis for $TTS^2$ along ${\mathcal P}$:
\[
E_{(1)}=2{\mathbb R}{ e}\left[\alpha_1\left(\frac{\partial}{\partial \xi}+\partial F\frac{\partial}{\partial \eta}
          +\partial \bar{F}\frac{\partial}{\partial \bar{\eta}}    \right) \right],
\]
\[
E_{(2)}=2{\mathbb R}{ e}\left[\alpha_2\left(\frac{\partial}{\partial \xi}+\partial F\frac{\partial}{\partial \eta}
          +\partial \bar{F}\frac{\partial}{\partial \bar{\eta}}    \right) \right],
\]
\[
E_{(3)}=2{\mathbb R}{ e}\left[\alpha_2\left(\frac{\partial}{\partial \xi}
       +(\bar{\partial} \bar{F}-2(F\partial u-\bar{F}\bar{\partial} u))\frac{\partial}{\partial \eta}
          -\partial \bar{F}\frac{\partial}{\partial \bar{\eta}}    \right) \right],
\]
\[
E_{(4)}=2{\mathbb R}{ e}\left[\alpha_1\left(\frac{\partial}{\partial \xi}
       +(\bar{\partial} \bar{F}-2(F\partial u-\bar{F}\bar{\partial} u))\frac{\partial}{\partial \eta}
          -\partial \bar{F}\frac{\partial}{\partial \bar{\eta}}    \right) \right],
\]
for
\[
\alpha_1=\frac{e^{-u-{\scriptstyle \frac{1}{2}}\phi i+{\scriptstyle \frac{1}{4}}\pi i}}
      {\sqrt{2}[-\lambda-|\sigma|]^{\scriptstyle \frac{1}{2}}}
\qquad\qquad
\alpha_2=\frac{e^{-u-{\scriptstyle \frac{1}{2}}\phi i-{\scriptstyle \frac{1}{4}}\pi i}}
      {\sqrt{2}[-\lambda+|\sigma|]^{\scriptstyle \frac{1}{2}}},
\]
where $\bar{\partial}F=-|\sigma|e^{-i\phi}$ and we have introduced $e^{2u}=4(1+\xi\bar{\xi})^{-2}$. Note that when $|\sigma|=0$, then $\phi$ is just a gauge freedom for the frame.

Moreover, $\{E_{(1)},E_{(2)}\}$ span $T{\mathcal P}$ and $\{E_{(3)},E_{(4)}\}$ span $N{\mathcal P}$. 
\end{Prop}

\vspace{0.1in}

Using the same notation as above:

\vspace{0.1in}

\begin{Prop}\label{p:dualbasis}
The dual basis of 1-forms is:
\[
\theta^{(1)}={\mathbb I}m\left[(\alpha_1\partial\bar{F}+\bar{\alpha}_1(\bar{\partial}\bar{F}
    -2(F\partial u-\bar{F}\bar{\partial} u)))d\xi-\bar{\alpha}_1d\eta\right]e^{2u},
\]
\[
\theta^{(2)}=\;{\mathbb I}m\left[(\alpha_2\partial\bar{F}+\bar{\alpha}_2(\bar{\partial}\bar{F}
    -2(F\partial u-\bar{F}\bar{\partial} u)))d\xi-\bar{\alpha}_2d\eta\right]e^{2u},
\]
\[
\theta^{(3)}=\;{\mathbb I}m\left[(\alpha_2\partial\bar{F}-\bar{\alpha}_2\partial F)d\xi
       +\bar{\alpha}_2d\eta\right]e^{2u},
\]
\[
\theta^{(4)}={\mathbb I}m\left[(\alpha_1\partial\bar{F}-\bar{\alpha}_1\partial F)d\xi
       +\bar{\alpha}_1d\eta\right]e^{2u}.
\]
\end{Prop}

\vspace{0.1in}

\begin{Def}
We refer to the above frame as the {\it canonical frame} associated with the positive plane ${\mathcal P}$. Any other frame that respects the tangent and normal splitting $TTS^2=T{\mathcal P}\oplus N{\mathcal P}$ is of the form
\[
E'_{(1)}=\cos\theta {E}_{(1)}-\sin\theta {E}_{(2)},
\qquad
E'_{(2)}=\sin\theta {E}_{(1)}+\cos\theta {E}_{(2)},
\]
\[
E'_{(3)}=\cos\psi {E}_{(3)}+\sin\psi {E}_{(4)},
\qquad
E'_{(4)}=-\sin\psi {E}_{(3)}+\cos\psi {E}_{(4)},
\]
for some $\theta,\psi\in S^1$.
\end{Def}
\vspace{0.1in}

Now consider the Levi-Civita connection $\overline{\nabla}$ associated with ${\mathbb G}$ and for $X,Y\in T{\mathcal P}$ we have the orthogonal splitting
\[
\overline{\nabla}_X Y= \nabla^\parallel_X Y+A(X,Y),
\]
where $A:T{\mathcal P}\times T{\mathcal P}\rightarrow N{\mathcal P}$ is the second fundamental form of the immersed plane ${\mathcal P}$.

Let $\Delta=\lambda^2-|\sigma|^2$, not to be confused with the Laplacian $\triangle$ of the last section.

\vspace{0.1in}

\begin{Prop}\label{p:2ndff}
The second fundamental form is:
\[
 A(e_{(a)},e_{(b)})=2{\mathbb R}{ e}\left[\beta_{ab}\left(\frac{\partial}{\partial \xi}
       +(\bar{\partial} \bar{F}-2(F\partial u-\bar{F}\bar{\partial} u))\frac{\partial}{\partial \eta}
          -\partial \bar{F}\frac{\partial}{\partial \bar{\eta}}    \right) \right],
\]
for $a,b=1,2$, where
\[
\beta_{11}=\left[i\lambda\partial |\sigma|-\sigma\bar{\partial}|\sigma|+i\lambda\partial\lambda-\sigma\bar{\partial}\lambda
    +|\sigma|(|\sigma|+\lambda)(\partial\phi-ie^{i\phi}\bar{\partial}\phi+2i\partial u-2e^{i\phi}\bar{\partial}u)\right]
\]
\[
\left/\left[2e^{2u+i\phi}(|\sigma|+\lambda)^2(-|\sigma|+\lambda)\right]\right.,
\]
\[
\beta_{22}=\left[-i\lambda\partial |\sigma|+\sigma\bar{\partial}|\sigma|+i\lambda\partial\lambda-\sigma\bar{\partial}\lambda
    +|\sigma|(|\sigma|-\lambda)(\partial\phi+ie^{i\phi}\bar{\partial}\phi+2i\partial u+2e^{i\phi}\bar{\partial}u)\right]
\]
\[
\left/\left[2 e^{2u+i\phi}(|\sigma|-\lambda)^2(-|\sigma|-\lambda)\right]\right.,
\]
\[
\beta_{12}=\left(-|\sigma|\partial |\sigma|+i\lambda e^{i\phi}\bar{\partial}|\sigma|+\lambda\partial\lambda-i\sigma\bar{\partial}\lambda
    \right)
\]
\[
\left/\left[2e^{2u+i\phi}(|\sigma|^2-\lambda^2)\sqrt{|\Delta|}\right]\right. .
\]
\end{Prop}
\begin{proof}
Consider the parallel and perpendicular projection operators $^\parallel P:TTS^2\rightarrow T{\mathcal P}$ and $^\perp P:TTS^2\rightarrow N{\mathcal P}$. These are given in terms of an adapted frame by
\[
^\parallel P_j^k=\delta_j^k-E_{(3)}^k\theta_j^{(3)}-E_{(4)}^k\theta_j^{(4)}
\qquad\qquad
^\perp P_j^k=\delta_j^k-E_{(1)}^k\theta_j^{(1)}-E_{(2)}^k\theta_j^{(2)}.
\]
The parallel projection operator has the following coordinate description:
\begin{align}
^\parallel P_{\bar{\eta}}^\xi=-{\textstyle{\frac{1}{2\Delta}}}\bar{\sigma}
&\qquad^\parallel P_{\bar{\xi}}^\xi=-{\textstyle{\frac{1}{2\Delta}}} (\bar{\partial}\bar{F}+\lambda i)\bar{\sigma},\nonumber\\
^\parallel P_\eta^\xi= -{\textstyle{\frac{1}{2\Delta}}} \lambda i
&\qquad^\parallel P_\xi^\xi= {\textstyle{\frac{1}{2\Delta}}} [(\partial F -2\lambda i)\lambda i -|\sigma|^2], \nonumber\\
^\parallel P_{\bar{\eta}}^\eta={\textstyle{\frac{1}{2\Delta}}} \bar{\sigma}(\partial F-\lambda i)
&\qquad^\parallel P_{\bar{\xi}}^\eta={\textstyle{\frac{1}{2\Delta}}} [-\bar{\sigma}[\partial F\bar{\partial}\bar{F}-|\sigma|^2
                                 -\lambda i(\bar{\partial}\bar{F}-\partial F)]+2\lambda^2],\nonumber\\
^\parallel P_\eta^\eta= -{\textstyle{\frac{1}{2\Delta}}}[\lambda i\partial F +|\sigma|^2]
&\qquad^\parallel P_\xi^\eta= {\textstyle{\frac{1}{2\Delta}}} \lambda i[(\partial F-2\lambda i)\partial F-|\sigma|^2]\nonumber,
\end{align}
while the perpendicular projection operator is
\[
^\perp P_\xi^\xi=\;^\parallel P_\eta^\eta
\qquad\qquad
^\perp P_{\bar{\xi}}^\xi=-\;^\parallel P_{\bar{\xi}}^\xi
\qquad\qquad
^\perp P_\eta^\xi=-\;^\parallel P_\eta^\xi
\qquad\qquad
^\perp P_{\bar{\eta}}^\xi=-\;^\parallel P_{\bar{\eta}}^\xi,
\]
\[
^\perp P_\xi^\eta=-\;^\parallel P_\xi^\eta
\qquad\qquad
^\perp P_{\bar{\xi}}^\eta=-\;^\parallel P_{\bar{\xi}}^\eta
\qquad\qquad
^\perp P_\eta^\eta=\;^\parallel P_\xi^\xi
\qquad\qquad
^\perp P_{\bar{\eta}}^\eta=-\;^\parallel P_{\bar{\eta}}^\eta.
\]
In terms of a frame in which $\{E_{(1)},E_{(2)}\}$ span the tangent space of ${\mathcal P}$, the second fundamental form is
\[
A_{(ab)}^{\;\;\;\;\;j}=\;^\perp P_k^j\;E_{(a)}^l\overline{\nabla}_l\;E_{(b)}^k.
\]
The result follows by direct computation of these quantities using Propositions \ref{p:frames} and \ref{p:dualbasis}.

\end{proof}

\vspace{0.1in}

\begin{Prop}
The mean curvature vector of the plane ${\mathcal P}$ is:
\[
H=2{\mathbb R}{ e}\left[\gamma\left(\frac{\partial}{\partial \xi}
       +(\bar{\partial} \bar{F}-2(F\partial u-\bar{F}\bar{\partial} u))\frac{\partial}{\partial \eta}
          -\partial \bar{F}\frac{\partial}{\partial \bar{\eta}}    \right) \right],
\]
where
\[
\gamma=\left[-\lambda(-i\lambda\partial |\sigma|+\sigma\bar{\partial}|\sigma|)
     -|\sigma|(i\lambda\partial\lambda-\sigma\bar{\partial}\lambda)-|\sigma|(|\sigma|^2-\lambda^2)(\partial\phi+2i\partial u)\right]
\]
\[
\left/\left[e^{2u+i\phi}(|\sigma|^2-\lambda^2)^2\right]\right. .
\]
\end{Prop}
\begin{proof}
The mean curvature vector of the plane ${\mathcal P}$ is the trace of the second fundamental form, which is
\[
H^j=A_{(11)}^{\;\;\;\;\;j}+ A_{(22)}^{\;\;\;\;\;j}.
\]
The result follows from computing this with the aid of the previous Proposition.
\end{proof}
\vspace{0.1in}

\begin{Note}
We can also write the mean curvature vector component (see \cite{gak8} for a variational derivation of this formula)
\begin{equation}\label{e:meanc}
H^\xi=\frac{2e^{-2u}}{\sqrt{|\lambda^2-|\sigma|^2|}}\left[
              ie^{-2u}\partial\left(\frac{\bar{\sigma}e^{2u}}{\sqrt{|\lambda^2-|\sigma|^2|}}\right)
       -\bar{\partial}\left(\frac{\lambda}{\sqrt{|\lambda^2-|\sigma|^2|}}\right)\right].
\end{equation}
\end{Note}

\begin{Cor}
A holomorphic graph has vanishing mean curvature.
\end{Cor}
\begin{proof}
This follows from inserting $\sigma=0$ in equation (\ref{e:meanc}).
\end{proof}
\vspace{0.1in}

\vspace{0.1in}

\subsection{Angles between positive surfaces}\label{s:angles}

Proposition 2 in \cite{Gk02} gives a reduction of the action of the group $O(n,m)$ by the maximal compact subgroup $O(n)\times O(m)$ leading to a matrix of angles between pairs of positive $n$-planes in indefinite ${\mathbb R}^{n+m}$. In the definite case, these angles were defined by Wong in the 1960's \cite{Wong}. This reduction carries over to $n+m$-dimensional manifolds and we now consider in more detail the case of intersecting positive surfaces in $TS^2$. 

The positive surfaces we have in mind are the flowing disc $f_s(D)$ and the boundary plane ${\mathcal P}$ intersecting along $f_s(\partial D)$. While we are working pointwise along this intersection, for ease of notation we drop any mention of the point of intersection.

Assume that $f_s(D)$ and ${\mathcal P}$ are positive and intersect at a point. Choose an orthonormal frame 
$E_{(\mu)}=\{e_{(1)},e_{(2)},f_{(1)},f_{(2)}\}$ and coframe $E^{(\mu)}=\{e^{(1)},e^{(2)},f^{(1)},f^{(2)}\}$ so that 
$\{e_{(a)}\}$ span the tangent plane $Tf_s(D)$, while 
$\{f_{(a)}\}$ span the normal plane $Nf_s(D)$. Similarly, let $\tilde{E}_{(\mu)}=\{\tilde{e}_{(1)},\tilde{e}_{(2)},\tilde{f}_{(1)},\tilde{f}_{(2)}\}$ and 
$\tilde{E}^{(\mu)}=\{\tilde{e}^{(1)},\tilde{e}^{(2)},\tilde{f}^{(1)},\tilde{f}^{(2)}\}$ be similar frames and coframes for ${\mathcal P}$. 

\begin{Def}
For frames as above define the $4\times 4$ matrix
\[
M_{(\mu)}^{\;\;(\nu)}=<E_{(\mu)},\tilde{E}^{(\nu)}>.
\]
\end{Def}
Then
\vspace{0.1in}

\begin{Prop}
There exist adapted frames such that the matrix $M$ has the form:
\[
M=\left(
     \begin{array}{cccc}
      \cosh A & 0 & 0 & \sinh A\\
       0 &  \cosh B & \sinh B  & 0 \\
       0 &  \sinh B & \cosh B & 0 \\
       \sinh A & 0 & 0 & \cosh A
     \end{array}\right),
\]
for hyperbolic angles $A,B\in{\mathbb R}$.
\end{Prop}
\begin{proof}
This is a special case of Proposition 2 in \cite{Gk02}.
\end{proof}

\vspace{0.1in}
 
\begin{Cor}\label{c:multi}
In the special case where $f_s(D)$ and ${\mathcal P}$ intersect along a curve the angle matrix reduces to
\[
M=\left(
     \begin{array}{cccc}
      1 & 0 & 0 & 0\\
       0 &  \cosh B & \sinh B  & 0 \\
       0 &  \sinh B & \cosh B & 0 \\
       0 & 0 & 0 & 1
     \end{array}\right).
\]
\end{Cor}

Here we have established the basic fact of neutral geometry that if two positive planes intersect on a line, then their normal planes must also intersect each other in a line. To pursue this further we need to relate the various angles between the tangent and normal planes of the two surfaces and their intersections, together with their complex slopes.

\begin{Prop} \label{p:hypang1}
Define $\mu=|\sigma|/|\lambda|$ and $\tilde{\mu}=|\tilde{\sigma}|/|\tilde{\lambda}|$. By positivity of the surfaces $\mu<1$ and $\tilde{\mu}<1$. Then
\begin{align}\label{e:hypangle}
\cosh B&=\frac{(1-\tilde{\mu}^2)(1+2\cos 2\theta\;\mu+\mu^2)+(1-\mu^2)(1+2\cos 2\tilde{\theta}\;\tilde{\mu}+\tilde{\mu}^2)}
{2(1-\tilde{\mu}^2)^{\scriptstyle{\frac{1}{2}}}(1-\mu^2)^{\scriptstyle{\frac{1}{2}}}(1+\cos 2\theta\;\mu)(1+\cos 2\tilde{\theta}\;\tilde{\mu})}\nonumber\\
&\qquad\qquad\qquad\qquad -\frac{\sin 2\theta\sin 2 \tilde{\theta}\;\mu\tilde{\mu}}{(1+\cos 2\theta\;\mu)(1+\cos 2\tilde{\theta}\;\tilde{\mu})}\nonumber\\
&=\frac{(1-\tilde{\mu}^2)(1+2\cos 2\psi\;\mu+\mu^2)+(1-\mu^2)(1+2\cos 2\tilde{\psi}\;\tilde{\mu}+\tilde{\mu}^2)}
  {2(1-\tilde{\mu}^2)^{\scriptstyle{\frac{1}{2}}}(1-\mu^2)^{\scriptstyle{\frac{1}{2}}}(1+\cos 2\psi\;\mu)(1+\cos 2\tilde{\psi}\;\tilde{\mu})}\nonumber\\
&\qquad\qquad\qquad\qquad -\frac{\sin 2\psi\sin 2 \tilde{\psi}\;\mu\tilde{\mu} }
{(1+\cos 2\psi\;\mu)(1+\cos 2\tilde{\psi}\;\tilde{\mu})},
\end{align}
where $\theta$,$\tilde{\theta}$,$\psi$ and $\tilde{\psi}$ determine the angles that the lines of intersection make with the canonical frames on $TD$,$T{\mathcal P}$,$ND$ and $N{\mathcal P}$, respectively.
\end{Prop}

\begin{proof}
Consider the tangent space to a point $\gamma\in D\cap{\mathcal P}$ (assumed to be non-empty). This can be split two distinct ways
\[
T\gamma TS^2=T_\gamma D\oplus N_\gamma D=T_\gamma {\mathcal P}\oplus N_\gamma {\mathcal P},
\]
and adapted orthonormal bases $\{e_{(a)},f_{(b)}\}_{a,b=1}^2$ and $\{\tilde{e}_{(a)},\tilde{f}_{(b)}\}_{a,b=1}^2$ chosen for the splittings (respectively).

Recall the canonical frames from Proposition \ref{p:frames}:
\[
e_{(1)}=2{\mathbb R}{ e}\left[\alpha_1\left(\frac{\partial}{\partial \xi}+\partial F\frac{\partial}{\partial \eta}
          +\partial \bar{F}\frac{\partial}{\partial \bar{\eta}}    \right) \right],
\]
\[
e_{(2)}=2{\mathbb R}{ e}\left[\alpha_2\left(\frac{\partial}{\partial \xi}+\partial F\frac{\partial}{\partial \eta}
          +\partial \bar{F}\frac{\partial}{\partial \bar{\eta}}    \right) \right],
\]
\[
f_{(1)}=2{\mathbb R}{ e}\left[\alpha_2\left(\frac{\partial}{\partial \xi}
       +(\bar{\partial} \bar{F}-2(F\partial u-\bar{F}\bar{\partial} u))\frac{\partial}{\partial \eta}
          -\partial \bar{F}\frac{\partial}{\partial \bar{\eta}}    \right) \right],
\]
\[
f_{(2)}=2{\mathbb R}{ e}\left[\alpha_1\left(\frac{\partial}{\partial \xi}
       +(\bar{\partial} \bar{F}-2(F\partial u-\bar{F}\bar{\partial} u))\frac{\partial}{\partial \eta}
          -\partial \bar{F}\frac{\partial}{\partial \bar{\eta}}    \right) \right],
\]
for
\begin{equation}\label{e:alphas}
\alpha_1=\frac{e^{-u-{\scriptstyle \frac{1}{2}}\phi i+{\scriptstyle \frac{1}{4}}\pi i}}
      {\sqrt{2}[-\lambda-|\sigma|]^{\scriptstyle \frac{1}{2}}},
\qquad\qquad
\alpha_2=\frac{e^{-u-{\scriptstyle \frac{1}{2}}\phi i-{\scriptstyle \frac{1}{4}}\pi i}}
      {\sqrt{2}[(-\lambda+|\sigma|)]^{\scriptstyle \frac{1}{2}}},
\end{equation}
where $\bar{\partial}F=-|\sigma|e^{-i\phi}$ and $e^{2u}=4(1+\xi\bar{\xi})^{-2}$.


Analogous expressions hold for the bases of $T{\mathcal P}$ and $N{\mathcal P}$, with a tilde on appropriate quantities. 

Assume that $D$ and ${\mathcal P}$ intersect along a curve and rotate the frames in the tangent bundle so that
the first vector of each basis coincide and lie along the tangent to the intersection. Thus, referring to the above frames, there exists  
rotated frames $\{\mathring{e}_{(a)},\mathring{f}_{(b)}\}$ and $\{\mathring{\tilde{e}}_{(a)},\mathring{\tilde{f}}_{(b)}\}$
with
\[
\mathring{e}_{(1)}=\cos\theta {e}_{(1)}-\sin\theta {e}_{(2)},
\qquad\qquad
\mathring{e}_{(2)}=\sin\theta {e}_{(1)}+\cos\theta {e}_{(2)},
\] 
\[
\mathring{\tilde{e}}_{(1)}=\cos\tilde{\theta} {\tilde{e}}_{(1)}-\sin\tilde{\theta} {\tilde{e}}_{(2)},
\qquad\qquad
\mathring{\tilde{e}}_{(2)}=\sin\tilde{\theta} {\tilde{e}}_{(1)}+\cos\tilde{\theta} {\tilde{e}}_{(2)},
\]
and $\mathring{e}_{(1)}=\mathring{\tilde{e}}_{(1)}$, for some $\theta,\tilde{\theta}\in [0,2\pi)$.

Now compute $\cosh B={\mathbb G}(\mathring{e}_{(2)},\mathring{\tilde{e}}_{(2)})$ with the aid of these expressions. First note that $\mathring{e}_{(1)}=\mathring{\tilde{e}}_{(1)}$ implies that
\begin{equation}\label{e:intsect1}
\cos\theta {\alpha}_{1}-\sin\theta {\alpha}_{2}=\cos\tilde{\theta} {\tilde{\alpha}}_{1}-\sin\tilde{\theta} {\tilde{\alpha}}_{2},
\end{equation}
\begin{equation}\label{e:intsect2}
(\cos\theta {\alpha}_{1}-\sin\theta {\alpha}_{2})\partial(F-\tilde{F})+(\cos\theta \bar{\alpha}_{1}
-\sin\theta \bar{\alpha}_{2})\bar{\partial}(F-\tilde{F})=0.
\end{equation}

Substituting the expressions for $\alpha_1$ and $\alpha_2$ in equation (\ref{e:intsect1}), we find
\begin{equation}\label{e:intsect2a}
e^{{\scriptstyle \frac{1}{2}}{\phi}}\left(\frac{\cos{\theta}}{[-{\lambda}-|{\sigma}|]^{\scriptstyle \frac{1}{2}}}
    -\frac{i\sin{\theta}}{[-{\lambda}+|{\sigma}|]^{\scriptstyle \frac{1}{2}}}\right)
  =e^{{\scriptstyle \frac{1}{2}}\tilde{\phi}}\left(\frac{\cos\tilde{\theta}}{[-\tilde{\lambda}-|\tilde{\sigma}|]^{\scriptstyle \frac{1}{2}}}
    -\frac{i\sin\tilde{\theta}}{[-\tilde{\lambda}+|\tilde{\sigma}|]^{\scriptstyle \frac{1}{2}}}\right).
\end{equation}
The norm and argument of this equation give
\begin{equation}\label{e:hypangcon1}
(\tilde{\lambda}^2-|\tilde{\sigma}|^2)(-\lambda+|\sigma|\cos 2\theta)=(\lambda^2-|\sigma|^2)(-\tilde{\lambda}+|\tilde{\sigma}|\cos 2\tilde{\theta}),
\end{equation}
and
\begin{equation}\label{e:hypangcon2}
e^{(\phi-\tilde{\phi})i}=\frac{(\lambda^2-|\sigma|^2)(|\tilde{\sigma}|-\tilde{\lambda}\cos 2\tilde{\theta}
   -i(\tilde{\lambda}^2-|\tilde{\sigma}|^2)^{\scriptstyle \frac{1}{2}} \sin 2\tilde{\theta})}
   {(\tilde{\lambda}^2-|\tilde{\sigma}|^2)(|{\sigma}|-{\lambda}\cos 2{\theta}
   -i({\lambda}^2-|{\sigma}|^2)^{\scriptstyle \frac{1}{2}} \sin 2{\theta})}.
\end{equation}

Moreover, introducing the complex slopes $\sigma$ and $\rho=\vartheta+i\lambda$ as before, equation (\ref{e:intsect2}) can be split into real and imaginary parts. The result, with the aid of equation (\ref{e:hypangcon2}), is
\begin{equation}\label{e:hypangcon3}
\vartheta-\tilde{\vartheta}=\frac{\sin 2\tilde{\theta}(\tilde{\lambda}^2-|\tilde{\sigma}|^2)^{\scriptstyle \frac{1}{2}}}
  {-\tilde{\lambda}+|\tilde{\sigma}|\cos 2\tilde{\theta}}|\tilde{\sigma}|
   -\frac{\sin 2{\theta}({\lambda}^2-|{\sigma}|^2)^{\scriptstyle \frac{1}{2}}}
  {-{\lambda}+|{\sigma}|\cos 2{\theta}}|{\sigma}|,
\end{equation}
and
\begin{equation}\label{e:hypangcon4}
\lambda-\tilde{\lambda}=-\frac{\tilde{\lambda}\cos 2\tilde{\theta}-|\tilde{\sigma}|}
{-\tilde{\lambda}+|\tilde{\sigma}|\cos 2\tilde{\theta}}|\tilde{\sigma}|
  +\frac{{\lambda}\cos 2{\theta}-|{\sigma}|}{-{\lambda}+|{\sigma}|\cos 2{\theta}}|{\sigma}|.
\end{equation}
Note that equations (\ref{e:hypangcon1}) and (\ref{e:hypangcon4}) are, in fact, the same equation.

We now compute, with the aid of the metric expression, that
\begin{align}
{\mathbb G}(\mathring{e}_{(2)},\mathring{\tilde{e}}_{(2)})&=\frac{4}{(1+\xi\bar{\xi})^2}{\mathbb I}{\mbox {m}}\left[  
     (\sin\theta \bar{\alpha}_{1}+\cos\theta \bar{\alpha}_{2})(\sin\tilde{\theta} \tilde{\alpha}_{1}+\cos\tilde{\theta} \tilde{\alpha}_{2})
     (\bar{\rho}-\tilde{\rho})\right. \nonumber\\
&\qquad \qquad\qquad \qquad\left. -(\sin\theta {\alpha}_{1}+\cos\theta {\alpha}_{2})(\sin\tilde{\theta} \tilde{\alpha}_{1}+\cos\tilde{\theta} \tilde{\alpha}_{2})(\tilde{\sigma}+\sigma)\right].\nonumber
\end{align}

The expression for the angle given in the Proposition follows from a lengthy computation in which equations (\ref{e:hypangcon2}) and (\ref{e:hypangcon3}), recalling that $\rho=\vartheta+i\lambda$,
along with the definitions of $\alpha_1$ and $\alpha_2$ given in equations (\ref{e:alphas}), are substituted in the above expression. These result in

\begin{align}\label{e:hypangle1}
\cosh B=&-\frac{(\lambda^2-|\sigma|^2)^{\scriptstyle{\frac{1}{2}}}}
{(\tilde{\lambda}^2-|\tilde{\sigma}|^2)^{\scriptstyle{\frac{1}{2}}}(-\lambda+|\sigma|\cos 2\theta)}
\left(\lambda+\frac{\tilde{\lambda}\cos 2\tilde{\theta}-|\tilde{\sigma}|}{|\tilde{\sigma}|\cos 2\tilde{\theta}-\tilde{\lambda}}|\tilde{\sigma}|\right)
\nonumber\\
&\qquad\qquad-\frac{\sin 2\theta\sin 2\tilde{\theta}|\sigma||\tilde{\sigma}|}
{(-\lambda+|\sigma|\cos 2\theta)(-\tilde{\lambda}+|\tilde{\sigma}|\cos 2\tilde{\theta})},
\end{align}
where the intersection constraint is
\begin{equation}\label{e:intersectionconstraint1}
(\tilde{\lambda}^2-|\tilde{\sigma}|^2)(-\lambda+|\sigma|\cos 2\theta)=(\lambda^2-|\sigma|^2)(-\tilde{\lambda}+|\tilde{\sigma}|\cos 2\tilde{\theta}).
\end{equation}

The first equation of the Proposition follows from using the definitions $\mu=|\sigma|/|\lambda|$ and $\tilde{\mu}=|\tilde{\sigma}|/|\tilde{\lambda}|$. Note that, the intersection equation can be taken as defining the ratio of the twists of the intersecting surfaces:

\begin{equation}\label{e:intersectionconstraint}
\frac{\tilde{\lambda}}{\lambda}=\frac{(1-\mu^2)(1+\cos 2\tilde{\theta}\;\tilde{\mu})}{(1-\tilde{\mu}^2)(1+\cos 2\theta\;\mu)}.
\end{equation}

The preceding arguments, which involve the angles $\theta$ and $\tilde{\theta}$, can also be expressed in terms of the angles $\psi$ and $\tilde{\psi}$ tracking the intersection of the {\it normal} bundles $ND$ and $N{\mathcal P}$, respectively.  

More specifically suppose that the adapted frames in the normal bundles are
\[
\mathring{f}_{(1)}=\cos\psi {f}_{(1)}+\sin\psi {f}_{(2)},
\qquad\qquad
\mathring{f}_{(2)}=-\sin\psi {f}_{(1)}+\cos\psi {f}_{(2)},
\] 
\[
\mathring{\tilde{f}}_{(1)}=\cos\tilde{\psi} {\tilde{f}}_{(1)}+\sin\tilde{\psi} {\tilde{f}}_{(2)},
\qquad\qquad
\mathring{\tilde{f}}_{(2)}=-\sin\tilde{\psi} {\tilde{f}}_{(1)}+\cos\tilde{\psi} {\tilde{f}}_{(2)},
\]
and $\mathring{f}_{(2)}=\mathring{\tilde{f}}_{(2)}$, for some $\psi,\tilde{\psi}\in [0,2\pi)$.

Then the angle can be computed using $\cosh B=-{\mathbb{G}}(\mathring{f}_{(1)},\mathring{\tilde{f}}_{(1)})$ which yields the second part of the 
Proposition.

We also have the relations ({\it c.f.} equations (\ref{e:hypangcon1}), (\ref{e:hypangcon2}) and (\ref{e:hypangcon3})):
\[
(\tilde{\lambda}^2-|\tilde{\sigma}|^2)(-\lambda+|\sigma|\cos 2\psi)=(\lambda^2-|\sigma|^2)(-\tilde{\lambda}+|\tilde{\sigma}|\cos 2\tilde{\psi}),
\]
\[
e^{(\phi-\tilde{\phi})i}=\frac{(\lambda^2-|\sigma|^2)(|\tilde{\sigma}|-\tilde{\lambda}\cos 2\tilde{\psi}
   -i(\tilde{\lambda}^2-|\tilde{\sigma}|^2)^{\scriptstyle \frac{1}{2}} \sin 2\tilde{\psi})}
   {(\tilde{\lambda}^2-|\tilde{\sigma}|^2)(|{\sigma}|-{\lambda}\cos 2{\psi}
   -i({\lambda}^2-|{\sigma}|^2)^{\scriptstyle \frac{1}{2}} \sin 2{\psi})},
\]
\[
\vartheta-\tilde{\vartheta}=-\frac{\sin 2\tilde{\psi}(\tilde{\lambda}^2-|\tilde{\sigma}|^2)^{\scriptstyle \frac{1}{2}}}
  {-\tilde{\lambda}+|\tilde{\sigma}|\cos 2\tilde{\psi}}|\tilde{\sigma}|
   +\frac{\sin 2{\psi}({\lambda}^2-|{\sigma}|^2)^{\scriptstyle \frac{1}{2}}}
  {-{\lambda}+|{\sigma}|\cos 2{\psi}}|{\sigma}|.
\]
\end{proof}
\vspace{0.1in}

We now state a series of corollaries to these relations that play a key role in controlling the flow of the edge.

\begin{Cor}
If $f_s(D)$ is holomorphic along $f_s(\partial D)$, then 
\[
\cosh B=\frac{{\mbox{Area}}_\Omega(T\tilde{\sigma})}{{\mbox{Area}}_{\mathbb{G}}(T\tilde{\sigma})}.
\]
\end{Cor}
\vspace{0.1in}

\begin{Cor}
The hyperbolic angle can also be written
\begin{align}
\cosh B&=\frac{|\lambda|^{\textstyle{\frac{1}{2}}}(1+2\mu\cos 2\theta+\mu^2)}
{2|\tilde{\lambda}|^{\textstyle{\frac{1}{2}}}(1+\mu\cos 2\theta)^{\textstyle{\frac{3}{2}}}(1+\tilde{\mu}\cos 2\tilde{\theta})^{\textstyle{\frac{1}{2}}}}
+\frac{|\tilde{\lambda}|^{\textstyle{\frac{1}{2}}}(1+2\tilde{\mu}\cos 2\tilde{\theta}+\tilde{\mu}^2)}
{2|\lambda|^{\textstyle{\frac{1}{2}}}(1+\mu\cos 2\theta)^{\textstyle{\frac{1}{2}}}(1+\tilde{\mu}\cos 2\tilde{\theta})^{\textstyle{\frac{3}{2}}}}\nonumber\\
&\qquad\qquad\qquad\qquad-\frac{\sin 2\theta\sin 2\tilde{\theta}}{(1+\mu\cos 2\theta)(1+\tilde{\mu}\cos 2\tilde{\theta})},\nonumber\\
&=\frac{|\lambda|^{\textstyle{\frac{1}{2}}}(1+2\mu\cos 2\psi+\mu^2)}
{2|\tilde{\lambda}|^{\textstyle{\frac{1}{2}}}(1+\mu\cos 2\psi)^{\textstyle{\frac{3}{2}}}(1+\tilde{\mu}\cos 2\tilde{\psi})^{\textstyle{\frac{1}{2}}}}
+\frac{|\tilde{\lambda}|^{\textstyle{\frac{1}{2}}}(1+2\tilde{\mu}\cos 2\tilde{\psi}+\tilde{\mu}^2)}
{2|\lambda|^{\textstyle{\frac{1}{2}}}(1+\mu\cos 2\psi)^{\textstyle{\frac{1}{2}}}(1+\tilde{\mu}\cos 2\tilde{\psi})^{\textstyle{\frac{3}{2}}}}\nonumber\\
&\qquad\qquad\qquad\qquad-\frac{\sin 2\psi\sin 2\tilde{\psi}}{(1+\mu\cos 2\psi)(1+\tilde{\mu}\cos 2\tilde{\psi})},\label{e:hypang2}
\end{align}
\end{Cor}
\begin{proof}
The two equality follows from Proposition \ref{p:hypang1} and equation (\ref{e:intersectionconstraint}) and the equivalent $\psi$ version.
\end{proof}\vspace{0.1in}

\begin{Cor}
The hyperbolic angle between two intersecting positive planes satisfies the inequalities
\[
\frac{1-\mu\tilde{\mu}}{(1-\mu^2)^{\scriptstyle{\frac{1}{2}}}(1-\tilde{\mu}^2)^{\scriptstyle{\frac{1}{2}}}}\leq\cosh B
  \leq\frac{1+\mu\tilde{\mu}}{(1-\mu^2)^{\scriptstyle{\frac{1}{2}}}(1-\tilde{\mu}^2)^{\scriptstyle{\frac{1}{2}}}}.
\]
\end{Cor}
\begin{proof}
These follow from maximizing and minimizing $\cosh B$ over $\theta$ and $\tilde{\theta}$.
\end{proof}

\vspace{0.1in}
\begin{Cor}\label{c:nulltogether}
Consider the set $U_B$ of pairs of intersecting positive planes with fixed hyperbolic angle $B$ in ${\mathbb R}^{2,2}$. Let $(P_s,\tilde{P}_s)\in U_B$ be a curve of pairs of planes for $s\in[0,1)$. Then $\lim_{s\rightarrow 1}P_s$ is degenerate iff $\lim_{s\rightarrow 1}\tilde{P}_s$ is degenerate.
\end{Cor}
\begin{proof}
From the left-hand inequality of the previous Corollary it is clear that, for fixed $B$, $\lim_{s\rightarrow 1}\mu=1$ iff 
$\lim_{s\rightarrow 1}\tilde{\mu}=1$.
\end{proof}

\vspace{0.1in}

\begin{Cor}\label{c:ineq2}
The following inequality holds for intersecting positive planes forming an angle $B>0$:
\begin{equation}\label{e:ineq2}
\frac{\tilde{\mu}-\cosh B\sinh B\;(1-\tilde{\mu}^2)}{\tilde{\mu}^2+\cosh^2 B\;(1-\tilde{\mu}^2)}\leq \mu
  \leq\frac{\tilde{\mu}+\cosh B\sinh B\;(1-\tilde{\mu}^2)}{\tilde{\mu}^2+\cosh^2 B\;(1-\tilde{\mu}^2)},
\end{equation}
which can also be written
\begin{equation}\label{e:ineq3}
\frac{\cosh B-\tilde{\mu}\sinh B}{(1-\tilde{\mu}^2)^{\scriptstyle{\frac{1}{2}}}}\leq\frac{1}{(1-\mu^2)^{\scriptstyle{\frac{1}{2}}}}\leq \frac{\cosh B+\tilde{\mu}\sinh B}{(1-\tilde{\mu}^2)^{\scriptstyle{\frac{1}{2}}}}.
\end{equation}
\end{Cor}
\vspace{0.1in}


\section{ Mean Curvature Flow in {$TS^2$}{}}\label{s:mcf}

Throughout we use the term positive surface to mean spacelike surface: the induced metric is positive definite.

\vspace{0.1in}

\subsection{The unparameterized flow}

We now investigate the initial boundary value problem with initial and boundary conditions (i) to (iv) as before. We consider the evolution when the flowing and boundary planes are graphs of sections of $TS^2\rightarrow S^2$. In this case it is most convenient to use the base to parameterize the surfaces. That is, we consider a flowing surface given by $\xi\mapsto(\xi,\eta=F_s(\xi,\bar{\xi}))$ and boundary plane given by $\xi\mapsto(\xi,\eta=\tilde{F}(\xi,\bar{\xi}))$. 

For the extraction of interior estimates it is sufficient to consider the unparameterized flow
\[
\frac{d f}{ds}^\bot=H, 
\]
in which the time derivative is projected perpendicular to the surface. It is well known that this can be achieved locally and we compute the explicit expressions for the flow of the complex function $F_s$. 

\begin{Prop}
For a positive graph in $TS^2$, the mean curvature flow is
\begin{align}
\frac{\partial F}{\partial s}=&g^{jk}\partial_j\partial_k F+\frac{i\bar{\sigma}}{\Delta}\left((\sigma\xi-\bar{\rho}\bar{\xi})
    (1+\xi\bar{\xi})+\bar{F}-\bar{\xi}^2F\right)\nonumber\\
&=\frac{(1+\xi\bar{\xi})^2}{2(\lambda^2-\sigma\bar{\sigma})}\left(-2\bar{\sigma}\partial\lambda-i\bar{\sigma}\bar{\partial}\sigma+2\lambda\partial\bar{\sigma}+i\sigma\bar{\partial}\bar{\sigma}
   +\frac{4i\bar{\sigma}(\sigma\xi+\lambda i\bar{\xi})}{1+\xi\bar{\xi}}\right)\label{e:floweq}.
\end{align}
\end{Prop}
\begin{proof}
Consider an evolving positive disc $f:D\times[0,s_0)\rightarrow TS^2$ such that 
$f_s(\xi,\bar{\xi})=(\xi,\bar{\xi},F_s(\xi,\bar{\xi}),\bar{F}_s(\xi,\bar{\xi}))$. Then
\[
\frac{\partial f}{\partial s}=\frac{\partial F}{\partial s}\frac{\partial }{\partial \eta}+\frac{\partial \bar{F}}{\partial s}\frac{\partial }{\partial \bar{\eta}}.
\]
Projecting onto the normal of $\tilde{\mathcal P}$
\begin{align}
\frac{\partial f}{\partial s}^\bot=&(^\bot P^\xi_\eta \dot{F}+\;^\bot P^\xi_{\bar{\eta}} \dot{\bar{F}})\frac{\partial }{\partial \xi}
    +(^\bot P^\eta_\eta \dot{F}+\;^\bot P^\eta_{\bar{\eta}} \dot{\bar{F}})\frac{\partial }{\partial \eta}\nonumber\\
&\qquad+(^\bot P^{\bar{\xi}}_{\bar{\eta}} \dot{\bar{F}}+\;^\bot P^{\bar{\xi}}_{\eta} \dot{F})\frac{\partial }{\partial \bar{\xi}}
    +(^\bot P^{\bar{\eta}}_{\bar{\eta}} \dot{\bar{F}}+\;^\bot P^{\bar{\eta}}_{\eta} \dot{F})\frac{\partial }{\partial \bar{\eta}}\nonumber,
\end{align}
and so the mean curvature flow is
\[
^\bot P^\xi_\eta \dot{F}+\;^\bot P^\xi_{\bar{\eta}} \dot{\bar{F}}=H^\xi,
\]
or from the expressions of the projection operators given in the proof of Proposition \ref{p:2ndff}
\[
\frac{\lambda i}{2\Delta}\dot{F}-\frac{\bar{\sigma}}{2\Delta} \dot{\bar{F}}=H^\xi.
\]
Combining this with its complex conjugate we have
\begin{equation}\label{e:fdot1}
\dot{F}=-2\lambda iH^\xi+2\bar{\sigma}H^{\bar{\xi}}.
\end{equation}
Using the expression (\ref{e:meanc}) for the mean curvature we get that
\begin{align}
H^\xi=&\frac{(1+\xi\bar{\xi})^2}{4\Delta^2}\Big[2\left(i\partial\bar{\sigma}-\bar{\partial}\lambda
    -\frac{2i\bar{\xi}\bar{\sigma}}{1+\xi\bar{\xi}}\right)\Delta\nonumber\\
  &\qquad\qquad\qquad-2i\lambda\bar{\sigma}\partial\lambda+i\sigma\bar{\sigma}\partial\bar{\sigma}+i\bar{\sigma}^2\partial \sigma+2\lambda^2\bar{\partial}\lambda
      -\lambda\sigma\bar{\partial}\bar{\sigma}-\lambda\bar{\sigma}\bar{\partial}\sigma\Big]\nonumber,
\end{align}
and the second equality stated in the Proposition follows from inserting this in equation (\ref{e:fdot1}).

To see that the first equality in the Proposition holds, we compute
\begin{align}
g^{jk}\partial_j\partial_k F&=\frac{(1+\xi\bar{\xi})^2}{2\Delta}\left(i\bar{\sigma}\partial^2F-2\lambda\partial\bar{\partial}F
     -i\sigma\bar{\partial}^2F \right)\nonumber\\
&=\frac{(1+\xi\bar{\xi})^2}{2\Delta}\left[i\bar{\sigma}\partial\left(\theta+i\lambda+\frac{2\bar{\xi}F}{1+\xi\bar{\xi}}\right)
    +2\lambda\partial\bar{\sigma}+i\sigma\bar{\partial}\bar{\sigma} \right]\nonumber\\
&=\frac{(1+\xi\bar{\xi})^2}{2\Delta}\left[-2\bar{\sigma}\partial\lambda-i\bar{\sigma}\bar{\partial}\sigma+i\sigma\bar{\partial}\bar{\sigma} 
   +2\lambda\partial\bar{\sigma}\right.\nonumber\\
&\qquad\qquad\qquad\qquad\left.+i\bar{\sigma}\left(\frac{2(\sigma\xi+\rho\bar{\xi})}{1+\xi\bar{\xi}}
   -\frac{2(\bar{F}-\bar{\xi}^2F)}{(1+\xi\bar{\xi})^2} \right)\right]\nonumber,
\end{align}
where we have used identity (\ref{e:id2}) in the more convenient form
\[
\partial\theta=i\partial\lambda-(1+\xi\bar{\xi})^2\partial\left(\frac{\bar{\sigma}}{(1+\xi\bar{\xi})^2}\right)-\frac{2F}{(1+\xi\bar{\xi})^2}.
\]
Thus 
\[
g^{jk}\partial_j\partial_k F+\frac{i\bar{\sigma}}{\Delta}\left((\sigma\xi-\bar{\rho}\bar{\xi})
    (1+\xi\bar{\xi})+\bar{F}-\bar{\xi}^2F\right)\qquad\qquad\qquad\qquad\qquad\qquad\qquad\qquad
\]
\[
=\frac{(1+\xi\bar{\xi})^2}{2(\lambda^2-\sigma\bar{\sigma})}\left(-2\bar{\sigma}\partial\lambda-i\bar{\sigma}\bar{\partial}\sigma
    +i\sigma\bar{\partial}\bar{\sigma}+2\lambda\partial\bar{\sigma}+\frac{4i\bar{\sigma}(\sigma\xi+\lambda i\bar{\xi})}{1+\xi\bar{\xi}}\right),
\]
as claimed.
\end{proof}

\subsection{Evolution equations for a positive surface}\label{s:evolution}

\begin{Prop}\label{p:shearflow}
Under the mean curvature flow the shear evolves by:
\[
\left(\frac{\partial }{\partial s}-{\mathbb G}^{jk}\partial_j\partial_k\right)\sigma=\frac{H_1(1+\xi\bar{\xi})^2+2H_2(1+\xi\bar{\xi})+2H_3}{2\Delta^2},
\]
where
\begin{align}
H_1=&-4\lambda \sigma \partial \lambda \bar{\partial} \lambda
-2i \lambda \bar{\sigma} \partial \lambda \partial \sigma
+2(\lambda^{2}+|\sigma|^2) \partial \lambda \bar{\partial} \sigma
+2i \lambda \sigma \partial \lambda \partial \bar{\sigma}
+2\lambda^{2} \bar{\partial} \lambda \partial \sigma \nonumber\\
&\qquad+2\sigma^{2} \bar{\partial} \lambda \partial \bar{\sigma}
+i \bar{\sigma}^{2}\left(\partial \sigma\right)^{2}
-2\lambda \bar{\sigma} \partial \sigma \bar{\partial} \sigma
-2\lambda \sigma \bar{\partial} \sigma \partial \bar{\sigma}
-i \sigma^{2}\left(\partial \bar{\sigma}\right)^{2} \nonumber,
\end{align}
\begin{align}
H_2&=-2 \sigma \partial \lambda\left(2 i \lambda\bar{\sigma} \bar{\xi}+ (\lambda^{2}+\sigma \bar{\sigma}) \xi \right)
+2 \sigma \bar{\partial} \lambda \left( \lambda^{2}-\sigma \bar{\sigma}\right) \bar{\xi}
+ \partial\sigma \left(i\bar{\sigma} \bar{\xi}(3 \lambda^{2} - \sigma \bar{\sigma}) +2 \lambda^{3} \xi \right)\nonumber\\
& \qquad\qquad
+ \bar{\partial}\sigma\left( i\sigma \xi-2 \lambda \bar{\xi} \right)(\lambda^{2} -\sigma \bar{\sigma})
+2 \partial \bar{\sigma} \left( i\sigma^{2} \bar{\sigma} \bar{\xi}+ \lambda\sigma^{2} \xi \right) \nonumber,
\end{align}
and
\[
H_3=-\sigma \left( i\sigma \xi^{2}-3 i\bar{\sigma} \bar{\xi}^{2}-4 \lambda\xi \bar{\xi} \right)(\lambda^{2} -\sigma \bar{\sigma}).
\]
In addition,
\[
\left(\frac{\partial }{\partial s}-{\mathbb G}^{jk}\partial_j\partial_k\right)\rho=\frac{H_4(1+\xi\bar{\xi})^3+H_5(1+\xi\bar{\xi})^2
     +H_6(1+\xi\bar{\xi})+H_7}{2(1+\xi\bar{\xi})\Delta^2},
\]
where
\begin{align}
H_4&=4\lambda \bar{\sigma} \left(\partial \lambda\right)^{2}
-2\bar{\sigma}^{2} \partial \lambda \partial \sigma
+2i \lambda \bar{\sigma} \partial \lambda \bar{\partial} \sigma
-4\lambda^{2} \partial \lambda \partial \bar{\sigma}
-2\sigma \bar{\sigma} \partial \lambda \partial \bar{\sigma}
-2i \lambda \sigma \partial \lambda \bar{\partial} \bar{\sigma} \nonumber\\
&\qquad -i \bar{\sigma}^{2}\partial \sigma \bar{\partial} \sigma
+2\lambda \bar{\sigma} \partial \sigma \partial \bar{\sigma}
+i \lambda^{2}\partial \sigma \bar{\partial} \bar{\sigma}
-i \lambda^{2}\bar{\partial} \sigma \partial \bar{\sigma}
+2\lambda \sigma \left(\partial \bar{\sigma}\right)^{2}
+i \sigma^{2}\partial \bar{\sigma} \bar{\partial} \bar{\sigma} \nonumber,
\end{align}
\begin{align}
H_5&=-4 \partial \lambda \left(2 i \lambda\sigma \bar{\sigma} \xi
- \lambda^{2} \bar{\sigma} \bar{\xi}
-\sigma \bar{\sigma}^{2} \bar{\xi} \right) +2\partial\sigma\left( i \lambda^{2} \bar{\sigma} \xi
+ i\sigma \bar{\sigma}^{2} \xi
-2 \lambda \bar{\sigma}^{2} \bar{\xi} \right) \nonumber\\
&\qquad +2\left(2 \partial \bar{\sigma} i \lambda^{2}\sigma \xi
-2 \partial \bar{\sigma} \lambda\sigma \bar{\sigma} \bar{\xi}
+ \bar{\partial} \bar{\sigma} i \lambda^{2}\sigma \bar{\xi}
- \bar{\partial} \bar{\sigma} i\sigma^{2} \bar{\sigma} \bar{\xi} \right)\nonumber,
\end{align}
\[
H_6=-4( i (\lambda^{2}-\sigma \bar{\sigma})+ \lambda \vartheta)(\lambda^{2}-\sigma \bar{\sigma}),
\]
and
\[
H_7=4i (F \sigma \xi-\bar{F} \bar{\sigma} \bar{\xi})(\lambda^{2}-\sigma \bar{\sigma}).
\]
\end{Prop}
\begin{proof}
The proofs of these statements follow from differentiation of the flow equation (\ref{e:floweq}). We illustrate this for the flow of
$\sigma$, leaving the flow of $\rho$ to the reader.

We start by splitting the expression into convenient terms:
\[
-\dot{\bar{\sigma}}=\bar{\partial}\dot{F}=E_1+E_2+E_3+E_4,
\]
where $E_1$ is second order in the derivatives of $\lambda$ and $\sigma$, $E_2$ and $E_3$ are the quadratic and linear first order terms, and $E_4$ is the zeroth 
order terms. We now compute each of these terms in turn.

So, differentiating equation (\ref{e:floweq}) we have
\[
E_1=\frac{(1+\xi\bar{\xi})^2}{2\Delta}\left(-2\bar{\sigma}\bar{\partial}\partial\lambda-i\bar{\sigma}\bar{\partial}\bar{\partial}\sigma+2\lambda\bar{\partial}\partial\bar{\sigma}+i\sigma\bar{\partial}\bar{\partial}\bar{\sigma}\right).
\]

At this point we exploit the 3-jet identity (\ref{e:id2}) which we write in the more favorable form
\[
\partial\bar{\partial}\lambda={\mathbb I}{\mbox m}\;\left[\bar{\partial}\bar{\partial}\sigma-\frac{2\xi}{1+\xi\bar{\xi}}\bar{\partial}\sigma
+\frac{2\xi^2}{(1+\xi\bar{\xi})^2}\sigma\right]-\frac{2\lambda}{(1+\xi\bar{\xi})^2}.
\]
Inserting this in the expression for $E_1$ yields
\[
E_1=\frac{(1+\xi\bar{\xi})^2}{2\Delta}\left(-i\bar{\sigma}\partial\partial\bar{\sigma}+2\lambda\partial\bar{\partial}\bar{\sigma}+i\sigma\bar{\partial}\bar{\partial}\bar{\sigma}
-\frac{2i\bar{\sigma}(\xi\bar{\partial}\sigma-\bar{\xi}\partial\bar{\sigma})}{1+\xi\bar{\xi}}
+\frac{2i\bar{\sigma}(\xi^2\sigma-\bar{\xi}^2\bar{\sigma}-2i\lambda)}{(1+\xi\bar{\xi})^2}
\right).
\]
The first three terms of this, noting the expression for $g^{-1}$ in Proposition \ref{p:ts2indmet}, are easily seen to be the 
rough Laplacian of  $-\bar{\sigma}$:  
\[
E_1=-g^{jk}\partial_j\partial_k\bar{\sigma}+\frac{i\bar{\sigma}}{\Delta}\left(
(\bar{\xi}\partial\bar{\sigma}-\xi\bar{\partial}\sigma)(1+\xi\bar{\xi})
+\xi^2\sigma-\bar{\xi}^2\bar{\sigma}-2i\lambda\right).
\]
We note that in the final sets of expressions, the lower order terms introduced into $E_1$ by the 3-jet identity will have to be added to $E_3$ 
and $E_4$.

Moving to the quadratic first order term, differentiating equation (\ref{e:floweq}) we compute that
\begin{align}
E_2=&\frac{(1+\xi\bar{\xi})^2}{2\Delta}\Big(-2\bar{\partial}\bar{\sigma}\partial\lambda-i\bar{\partial}\bar{\sigma}\bar{\partial}\sigma+2\bar{\partial}\lambda\partial\bar{\sigma}+i\bar{\partial}\sigma\bar{\partial}\bar{\sigma}\nonumber\\
&\qquad\qquad\qquad\qquad
-\frac{1}{\Delta}\bar{\partial}\Delta(-2\bar{\sigma}\partial\lambda-i\bar{\sigma}\bar{\partial}\sigma+2\lambda\partial\bar{\sigma}+i\sigma\bar{\partial}\bar{\sigma})\Big)\nonumber\\
&=\frac{(1+\xi\bar{\xi})^2}{2\Delta^2}\Big(
4\lambda\bar{\sigma}\bar{\partial}\lambda\partial\lambda
+2i\lambda\bar{\sigma}\bar{\partial}\lambda\bar{\partial}\sigma
-2(\lambda^2+\sigma\bar{\sigma})\bar{\partial}\lambda\partial\bar{\sigma}
-2i\lambda\sigma\bar{\partial}\lambda\bar{\partial}\bar{\sigma}\nonumber\\
&\qquad\qquad\qquad\qquad
-2\lambda^2\bar{\partial}\bar{\sigma}\partial\lambda
+2\lambda\sigma\bar{\partial}\bar{\sigma}\partial\bar{\sigma}
+i\sigma^2(\bar{\partial}\bar{\sigma})^2-i\bar{\sigma}^2(\bar{\partial}\sigma)^2\nonumber\\
&\qquad\qquad\qquad\qquad\qquad\qquad\qquad\qquad
-2\bar{\sigma}^2\bar{\partial}\sigma\partial\lambda
+2\lambda\bar{\sigma}\bar{\partial}\sigma\partial\bar{\sigma}\Big)\nonumber\\
&=-\frac{(1+\xi\bar{\xi})^2}{2\Delta^2}\bar{H}_1\nonumber.
\end{align}
This establishes the quadratic first order term, once we recall that $\bar{\partial}\dot{F}=-\dot{\bar{\sigma}}$. 

Moving to the linear first order term
\begin{align}
E_3=&\frac{(1+\xi\bar{\xi})\xi}{\Delta}\left(-2\bar{\sigma}\partial\lambda-i\bar{\sigma}\bar{\partial}\sigma+2\lambda\partial\bar{\sigma}+i\sigma\bar{\partial}\bar{\sigma}\right)\nonumber\\
&\qquad\qquad
+\frac{1+\xi\bar{\xi}}{\Delta}\left(2i(\sigma\xi+i\lambda\bar{\xi})\bar{\partial}\bar{\sigma}+2i\bar{\sigma}(\xi\bar{\partial}\sigma+i\bar{\xi}\bar{\partial}\lambda)\right)\nonumber\\
&\qquad\qquad\qquad
-\frac{1+\xi\bar{\xi}}{\Delta^2}\left[2i\bar{\sigma}(\sigma\xi+i\lambda\bar{\xi})(2\lambda\bar{\partial}-\sigma\bar{\partial}\bar{\sigma}-\bar{\sigma}\bar{\partial}\sigma)\right]\nonumber\\
=&\frac{1+\xi\bar{\xi}}{\Delta^2}\Big\{-2\xi\bar{\sigma}(\lambda^2-\sigma\bar{\sigma})\partial\lambda+
[2\bar{\sigma}^2(i\sigma\xi-\lambda\bar{\xi})+i\xi\bar{\sigma}(\lambda^2-\sigma\bar{\sigma})]\bar{\partial}\sigma\nonumber\\
&\qquad\qquad\qquad
+2\xi\lambda(\lambda^2-\sigma\bar{\sigma})\partial\bar{\sigma}+[i\xi\sigma(3\lambda^2-\sigma\bar{\sigma})-2\bar{\xi}\lambda^3]\bar{\partial}\bar{\sigma}
\nonumber\\
&\qquad\qquad\qquad
+2[\bar{\xi}\bar{\sigma}(\lambda^2+\sigma\bar{\sigma})-2i\xi\lambda\sigma\bar{\sigma}]\bar{\partial}\lambda\Big\}\nonumber.
\end{align}
Adding the linear term from $E_1$ we compute that
\begin{align}
-\bar{H}_3=&\frac{\Delta^2}{1+\xi\bar{\xi}}E_3+i\bar{\xi}\bar{\sigma}(\lambda^2-\sigma\bar{\sigma})\partial\bar{\sigma}
  -i\xi\bar{\sigma}(\lambda^2-\sigma\bar{\sigma})\bar{\partial}\sigma\nonumber\\
=&\frac{1+\xi\bar{\xi}}{\Delta^2}\Big\{-2\xi\bar{\sigma}(\lambda^2-\sigma\bar{\sigma})\partial\lambda+
2\bar{\sigma}^2(i\sigma\xi-\lambda\bar{\xi})\bar{\partial}\sigma\nonumber\\
&\qquad\qquad\qquad
+(\lambda^2-\sigma\bar{\sigma})(2\xi\lambda+i\bar{\xi}\bar{\sigma})\partial\bar{\sigma}+[i\xi\sigma(3\lambda^2-\sigma\bar{\sigma})-2\bar{\xi}\lambda^3]\bar{\partial}\bar{\sigma}
\nonumber\\
&\qquad\qquad\qquad
+2[\bar{\xi}\bar{\sigma}(\lambda^2+\sigma\bar{\sigma})-2i\xi\lambda\sigma\bar{\sigma}]\bar{\partial}\lambda\Big\}\nonumber,
\end{align}
as claimed in the Proposition.

Finally, we work out the zero order term by looking again at the derivative of equation (\ref{e:floweq}) :
\[
E_4=\frac{2i\bar{\sigma}}{\Delta}\Big(\sigma\xi^2+i\lambda(1+2\xi\bar{\xi})\Big),
\]
and taking into account the zero order term of $E_1$, we have
\[
-\bar{H}_4=\Delta^2\;E_4+i\bar{\sigma}(\xi^2\sigma-\bar{\xi}^2\bar{\sigma}-2i\lambda)\Delta
   =\bar{\sigma}(3i\xi^2\sigma-i\bar{\xi}^2\bar{\sigma}-4\xi\bar{\xi}\lambda)\Delta,
 \]
as claimed.
\end{proof}

\vspace{0.1in}

\begin{Cor}\label{c:flowdet}
Under the mean curvature flow the determinant of the induced metric evolves by:
\[
\left(\frac{\partial }{\partial s}-{\mathbb G}^{jk}\partial_j\partial_k\right)\Delta=\frac{H_8(1+\xi\bar{\xi})^2}
{2\Delta^2}+\frac{H_9(1+\xi\bar{\xi})+H_{10}}{\Delta}-4\lambda,
\]
where
\begin{align}
H_8&=2{\mathbb R}{\mbox{e}}[2i(\sigma \bar{\sigma}-3 \lambda^{2}) \bar{\sigma} \left(\partial \lambda\right)^{2} 
+2\lambda(\lambda^{2}+\sigma \bar{\sigma}) \partial \lambda \bar{\partial} \lambda 
+4i \lambda \bar{\sigma}^{2} \partial \lambda \partial \sigma  \nonumber\\
&\qquad-4\sigma \bar{\sigma}^{2} \partial \lambda \bar{\partial} \sigma
+4i \lambda^{3} \partial \lambda \partial \bar{\sigma} 
-4\lambda^{2} \sigma \partial \lambda \bar{\partial} \bar{\sigma}
-i \bar{\sigma}^{3}\left(\partial \sigma\right)^{2} 
+2\lambda \bar{\sigma}^{2} \partial \sigma \bar{\partial} \sigma  \nonumber\\
&\qquad -2i \sigma \bar{\sigma}^{2} \partial \sigma \partial \bar{\sigma} 
+\lambda \sigma \bar{\sigma} \partial \sigma \bar{\partial} \bar{\sigma} 
+i (2\lambda^{2}-\sigma \bar{\sigma}) \bar{\sigma} \left(\bar{\partial} \sigma\right)^{2}
-2\lambda^{3} \bar{\partial} \sigma \partial \bar{\sigma} 
+3\lambda \sigma \bar{\sigma} \bar{\partial} \sigma \partial \bar{\sigma}], \nonumber
\end{align}
\[
H_9={\mathbb R}{\mbox{e}}\left[(-2i\lambda\bar{\xi}-4\sigma\xi)\bar{\sigma}\partial\lambda
    -i\bar{\sigma}^2\bar{\xi}\partial\sigma+(-3i\sigma\xi+4\lambda\bar{\xi})\bar{\sigma}\bar{\partial}\sigma\right]/2,
\]
and
\[
H_{10}=2{\mathbb R}{\mbox{e}}\;\sigma\bar{\sigma}\left[i(\sigma\xi^2-\bar{\sigma}\bar{\xi}^2)-2\lambda\xi\bar{\xi}\right].
\]

\end{Cor}
\begin{proof}
This follows from the previous evolution equations and the definition $\Delta=\lambda^2-|\sigma|^2$.
\end{proof}
\vspace{0.1in}

\begin{Prop}\label{p:chiflow}
Under mean curvature flow the perpendicular distance function $\chi$ (see Definition \ref{d:perpdist}) evolves by
\begin{align}
\left(\frac{\partial }{\partial s}-{\mathbb G}^{jk}\partial_j\partial_k\right)\chi^2&
  =4\left[i\bar{F}^2\bar{\sigma}-2F\bar{F}\lambda-iF^2\sigma\right.\nonumber\\
&\qquad\qquad
   +[i\sigma\xi(\bar{F}\bar{\sigma}-F\bar{\rho})-i\bar{\sigma}\bar{\xi}(F\sigma-\bar{F}\rho)](1+\xi\bar{\xi})\nonumber\\
&\qquad\qquad\qquad+\left.\lambda(\rho\bar{\rho}-\sigma\bar{\sigma})(1+\xi\bar{\xi})^2\right]/[(1+\xi\bar{\xi})^2(\lambda^2-\sigma\bar{\sigma})].\nonumber
\end{align}
\end{Prop}
\begin{proof}
This follows from differentiating the expression for $\chi(\xi,\bar{\xi},F,\bar{F})$ and using the equations for the flow of $F$.
\end{proof}
\vspace{0.1in}

\begin{Prop}
Under mean curvature flow the weighted shear evolves by
\begin{align}
\left(\frac{\partial }{\partial s}-{\mathbb G}^{jk}\partial_j\partial_k\right)\left(\frac{|\sigma|^2}{\lambda^2-|\sigma|^2}\right)&=-2\frac{(\lambda^2+|\sigma|^2)}{(\lambda^2-|\sigma|^2)^3}
    \Big\|\lambda d|\sigma|-|\sigma|d\lambda\Big\|^2
     -2\frac{\lambda^2|\sigma|^2}{(\lambda^2-|\sigma|^2)^2}
    \Big\|d\phi\Big\|^2\nonumber\\
&\qquad +\frac{ 4\lambda|\sigma| }{(\lambda^{2}-|\sigma|^2)^3} \;  {\mathbb R}{\mbox{e}}\;[H_{11}]  \label{e:holconv},
\end{align}
where
\begin{align}
H_{11}=&(1+ \xi \bar{\xi})\left[i\lambda\bar{\sigma} \bar{\xi}\partial|\sigma|-i|\sigma|\bar{\sigma} \bar{\xi}\partial \lambda
+2\lambda|\sigma|( i \lambda \xi- \bar{\sigma} \bar{\xi}) \partial\phi\right]\nonumber\\
&\qquad -2 i\lambda \sigma|\sigma| \xi^{2}-2|\sigma|^{3}+2( 1+2 \xi \bar{\xi})\lambda^{2}|\sigma|  \nonumber,
\end{align}
and $\phi$ is the argument of $\sigma$. Here the norm $\|.\|$ is taken with respect to the induced metric given in Proposition \ref{p:ts2indmet}.
\end{Prop}
\begin{proof}
For the sake of brevity, introduce the heat operator ${\mathfrak P}$:
\[
{\mathfrak P}=\frac{\partial }{\partial s}-{\mathbb G}^{jk}\partial_j\partial_k
\]
Then the result follows from the fact that
\begin{align}
{\mathfrak P}\left(\frac{|\sigma|^2}{\lambda^2-|\sigma|^2}\right)&=\frac{\lambda}{(\lambda^2-|\sigma|^2)^2}\left[\lambda\sigma{\mathfrak P}(\bar{\sigma})
   +\lambda\bar{\sigma}{\mathfrak P}(\sigma)
   -2|\sigma|^2{\mathfrak P}(\lambda)\right]\nonumber\\
&\qquad -\frac{2|\sigma|^2(3\lambda^2+  |\sigma|^2)}{(\lambda-|\sigma|^2)^3}\|d\lambda\|^2
     -\frac{2\lambda^2(\lambda^2+  |\sigma|^2)}{(\lambda-|\sigma|^2)^3}<<d\sigma,d\bar{\sigma}>>\nonumber\\
&\qquad  +\frac{4\lambda(\lambda^2+|\sigma|^2)}{(\lambda-|\sigma|^2)^3}<<\sigma d\bar{\sigma}+\bar{\sigma}d\sigma,d\lambda>> 
   -\frac{2\lambda^2\sigma^2}{(\lambda-|\sigma|^2)^2}\|d\bar{\sigma}\|^2\nonumber\\
&\qquad -\frac{2\lambda^2\bar{\sigma}^2}{(\lambda-|\sigma|^2)^2}\|d\sigma\|^2\nonumber,
\end{align}
and the flow equations given in Proposition \ref{p:shearflow}, recalling that $\lambda={\mathbb I}{\mbox{m}}\;\rho$. 
\end{proof}

\vspace{0.1in}

\subsection{Boundary conditions}\label{s:bdryinit}

In our case we would like the boundary plane to be the totally real Lagrangian hemisphere ${\mathcal P}$, but, as the metric will be Lorentz or degenerate on such a plane (see Proposition \ref{p:indmet}), a Lagrangian surface can never be positive and so cannot be used as a boundary condition. Instead we perturb the plane to make it positive, and attach the initial disc to this perturbed plane.

More specifically, suppose ${\mathcal P}$ is given by $\eta=F(\xi,\bar{\xi})$.
Define the perturbed plane $\tilde{\mathcal P}_{C_0}$ by adding a {\it linear holomorphic twist}: 
\begin{equation}\label{e:addtwist}
\eta=\tilde{F}=F-iC_0\xi,
\end{equation}
where $C_0$ is a real positive constant.

\begin{Prop}\label{p:linholtwist}
For any closed $K\subset\{|\xi|<1\}$, there exists $C_0>0$ such that $\tilde{\mathcal P}_{C}$ is positive on $K$ for all $C>C_0$ .
As we make $C_0$ large, the positive area containing $\xi=0$ becomes bigger, and tends to an open hemisphere as $C_0\rightarrow\infty$.

\end{Prop}
\begin{proof}
Since the deformation is holomorphic, the shear remains the same: $\tilde{\sigma}=\sigma$ and, computing the twist of $\tilde{\sigma}$
\[
\tilde{\lambda}={\mathbb I}{\mbox m}\;(1+\xi\bar{\xi})^2\partial\left[\frac{\tilde{F}}{(1+\xi\bar{\xi})^2}\right]
={\mathbb I}{\mbox m}\;(1+\xi\bar{\xi})^2\partial\left[\frac{F-iC_0\xi}{(1+\xi\bar{\xi})^2}\right]=-C_0\frac{1-\xi\bar{\xi}}{1+\xi\bar{\xi}},
\]
where we have used the fact that ${\mathcal P}$ is Lagrangian. For $C_0>|\tilde{\sigma}(0)|$ 
\[
\left.\left(\tilde{\lambda}^2-|\tilde{\sigma}|^2\right)\right|_0=C_0^2-|\tilde{\sigma}(0)|^2>0,
\]
i.e. the metric at $0$ is positive definite.
\end{proof}

\vspace{0.1in}

\begin{Def}
Fix $C_0>|\tilde{\sigma}(0)|$ and denote the set on which the induced metric is positive by
\[
\tilde{\Lambda}_{C_0}'=\left\{\;\gamma\in\tilde{\mathcal P}\;\;\left|\;\; |\tilde{\sigma}(\gamma)|<|\tilde{\lambda}(\gamma)|\;\right.\right\}.
\] 
Clearly $\tilde{\Lambda}_{C_0}'$ is non-empty since it contains $\gamma_0$. Denote the connected component of $\gamma_0$ in $\tilde{\Lambda}_{C_0}'$ 
by $\tilde{\Lambda}_{C_0}$. 
\end{Def}

\vspace{0.1in}

\begin{Note}
In order for the induced metric to be positive (rather than negative) definite we have arranged that $\tilde{\lambda}<0$ - see Proposition \ref{p:ts2indmet}.
\end{Note}
\vspace{0.1in}


\section{{ Existence Results for the I.B.V.P.}}\label{s:ibvp}

In this section we establish sufficient conditions, namely smallness of the initial angle and aholomorphicity along the edge, for which
long-time existence of the {\bf I.B.V.P.} holds. We then prove the existence of a holomorphic disc with edge lying on a totally real
Lagrangian hemisphere.

\subsection{Short-time existence}\label{ss:ste}

Short-time existence of quasilinear parabolic equations of various types can be established by comparison with Schauder theory for linear equations.  
In what follows we prove this fact for the {\bf I.B.V.P.} 

Consider then a 1-parameter family of maps over a fixed domain $f_s:D\rightarrow TS^2$ which has local coordinate expression $\nu\mapsto(\xi(s,\nu,\bar{\nu}),\eta(s,\nu,\bar{\nu}))$. Let $\tilde{\mathcal P}$ be a positive plane in $TS^2$ which is given as a graph 
$\eta=\tilde{F}(\xi,\bar{\xi})$ and $D_0$ some initial positive surface with edge lying in $\tilde{\mathcal P}$.

We can now establish short-time existence for the {\bf I.B.V.P.}:

\begin{Thm}\label{t:ste}
Let $f_0:D\rightarrow TS^2$ be a smooth positive section whose edge lies in a fixed positive section $\tilde{\sigma}$. Assume, in addition, that 
$\tilde{\sigma}$ is totally real along the edge $f_0(\partial D)$.
 
Then there exists a unique family of positive sections $f_s(D)$ with
\[
f_s\in C^{2+\alpha}(\overline{D}\times[0,s_0)),
\]
satisfying the {\bf I.B.V.P.} on an interval $0\leq s< s_0$. 
\end{Thm}

\vspace{0.1in}

\begin{proof}
We are dealing with a parabolic system with mixed non-linear boundary conditions. Short time existence for mean curvature flow in higher codimension can be found in \cite{smock} (Proposition 3.2) and references therein, while the codimension one boundary case can be found in Theorem 1.3 of \cite{stahl}. We will follow \cite{Freire} and verify that the Lopatinskii-Shapiro conditions hold for our boundary system.  

Following section I.2.3 in \cite {Eid}, first consider the linearization of the system at $f=(\xi,\eta)$, namely:
\[
\left(\frac{\partial }{\partial s}-a_{\alpha\beta}(\nabla f)D^\alpha D^\beta\right)\hat{f}=g,
\]
where the linearized initial and boundary conditions for $\hat{f}=(\hat{\xi},\hat{\eta})$ are
\begin{enumerate}
\item[(i)]  $\hat{f}(s=0)=f_0$,
\item[(ii)] $\hat{\eta}=\delta\tilde{F}$,
\item[(iii)] $4{\mathbb I}m(\bar{\beta}\partial_v\hat{\xi}+\alpha\partial_v\bar{\hat{\eta}})=0$,
\item[(iv)] ${\mathbb R}e [Ce^{-i\phi}(\partial_u\hat{\xi}\partial_v\eta-\partial_v\hat{\xi}\partial_u\eta
    +\partial_u\xi\partial_v\hat{\eta}-\partial_v\xi\partial_u\hat{\eta})]=0$,
\end{enumerate}
and, for brevity, we have introduced
\begin{equation}\label{e:def1}
Ce^{i\phi}=\partial_u\xi\partial_v\eta-\partial_v\xi\partial_u\eta.
\end{equation}

To show that this is a parabolic system we must check that the Lopatinskii-Shapiro conditions (equations (2.18) to (2.20) of \cite{Eid}) hold. We do this as follows.

Fix a point $p\in\partial D$ and, by an isometry of $TS^2$, set $\xi(f(p))=\eta(f(p))=0$. We retain the 
freedom to rotate $(\xi,\eta)\rightarrow (e^{i\theta}\xi,e^{-2i\theta}\eta)$, which we will implement shortly.

In addition, let us choose a parameterization $\nu=u+iv$ of the domain $D$ such that at the point $p$, the edge is given by $v=0$, and 
\[
\mathring{e}_{(1)}=f_*\left(\frac{\partial}{\partial u}\right),
\qquad\qquad
\mathring{e}_{(2)}=f_*\left(\frac{\partial}{\partial v}\right),
\]
form an orthonormal basis for $T_{f(p)}f(D)$ and
\[
\triangle_p=\left.\frac{\partial^2}{\partial u^2}\right|_p+\left.\frac{\partial^2}{\partial v^2}\right|_p.
\]
Note that orthonormality of the frame at $p$ implies that
\begin{equation}\label{e:onf1}
{\textstyle{\frac{1}{2i}}}(\partial_u\xi\partial_u\bar{\eta}-\partial_u\bar{\xi}\partial_u\eta)=1,
\end{equation}
\begin{equation}\label{e:onf2}
{\textstyle{\frac{1}{2i}}}(\partial_v\xi\partial_v\bar{\eta}-\partial_v\bar{\xi}\partial_v\eta)=1,
\end{equation}
\begin{equation}\label{e:onf3}
\partial_u\xi\partial_v\bar{\eta}-\partial_u\bar{\xi}\partial_v\eta-\partial_v\bar{\xi}\partial_u\eta+\partial_v\xi\partial_u\bar{\eta}=0.
\end{equation}
An orthonormal frame for $T_{f(p)}\tilde{\mathcal P}$ is given by
\[
\tilde{e}_1=e_1,
\qquad\qquad
\tilde{e}_2=2{\mathbb{R}}e\left(\alpha\frac{\partial}{\partial \xi}+\beta\frac{\partial}{\partial \eta}\right),
\]
for some $\alpha,\beta\in{\mathbb C}$ where positivity of $\tilde{\mathcal P}$ implies that $\beta\neq 0$. The fixed angle condition is 
\[
{\mathbb G}(e_2,\tilde{e}_2)=4{\mathbb I}m(\bar{\beta}\partial_v\xi+\alpha\partial_v\bar{\eta}) =\cosh B.
\]
On the other hand, the asymptotic holomorphicity boundary condition is
\[
|\partial_u\xi\partial_v\eta-\partial_v\xi\partial_u\eta|^2=C^2(1+s)^{-2}.
\]

Take the Fourier transform in the variable $u$ and the Laplace transform in the variable $s$. If $w$ and $t$ are the transformed variables,
respectively, we get the transformed ODE's
\[
\frac{\partial^2 \hat{\xi}}{\partial v^2}-(t+w^2)\hat{\xi}=0,
\qquad\qquad
\frac{\partial^2 \hat{\eta}}{\partial v^2}-(t+w^2)\hat{\eta}=0.
\]
We must now verify that these equations have a solution that decay for $v\rightarrow-\infty$ and satisfy the transformed boundary conditions
\begin{enumerate}
\item[(ii)]  $\hat{\eta}(0)=\delta\tilde{F}$,
\item[(iii)] $4{\mathbb I}m(\bar{\beta}\partial_v\hat{\xi}+\alpha\partial_v\bar{\hat{\eta}})(0)=0$,
\item[(iv)]  
\begin{align}
&Ce^{-i\phi}(iw\hat{\xi}\partial_v\eta-\partial_v\hat{\xi}\partial_u\eta+\partial_u\xi\partial_v\hat{\eta}-iw\partial_v\xi\hat{\eta})(0)\nonumber\\
&\qquad+Ce^{i\phi}(iw\bar{\hat{\xi}}\partial_v\bar{\eta}-\partial_v\bar{\hat{\xi}}\partial_u\bar{\eta}
           +\partial_u\bar{\xi}\partial_v\bar{\hat{\eta}}-iw\partial_v\bar{\xi}\bar{\hat{\eta}})(0)=0.\nonumber
\end{align}
\end{enumerate}
To do this, first solve the ODE's
\[
\hat{\xi}=\hat{\xi}(0)\exp(v\sqrt{t+w^2}),
\qquad\qquad
\hat{\eta}=\hat{\eta}(0)\exp(v\sqrt{t+w^2}),
\]
and then substituting this in the boundary conditions (i) to (iii), we are led to the following linear system for 
$\hat{U}=[\hat{\xi}(0),\bar{\hat{\xi}}(0),\hat{\eta}(0),\bar{\hat{\eta}}(0)]^T$:
\[
\hat{M}\hat{U}=\hat{V},
\]
where
\[
\hat{M}=\left[
\begin{array}{cccc}
0&0&1&0 \\
0&0&0&1 \\
\bar{\beta}&-\beta&-\bar{\alpha}&\alpha \\
Ce^{-i\phi}(iw\partial_v\eta-\sqrt{t+w^2}\partial_u\eta)&Ce^{i\phi}(iw\partial_v\bar{\eta}-\sqrt{t+w^2}\partial_u\bar{\eta})& * & * \\
\end{array}
\right].
\]
Thus the linearized system has a unique solution iff the determinant of this matrix is non-zero. Now
\[
|\hat{M}|=\bar{\beta}Ce^{i\phi}(iw\partial_v\bar{\eta}-\sqrt{t+w^2}\partial_u\bar{\eta})
      +\beta Ce^{-i\phi}(iw\partial_v\eta-\sqrt{t+w^2}\partial_u\eta).
\]
Clearly a necessary condition for this expression to be non-zero is that $C\neq0$. However, this is true by our asymptotic holomorphicity
condition, so long as it is true initially. Thus the edge of the initial disc must be totally real - which we now assume. 

Motivated by this expression, let us now exhaust our coordinate freedom by using the rotation $(\xi,\eta)\rightarrow (e^{i\theta}\xi,e^{-2i\theta}\eta)$
to set (recalling definition (\ref{e:def1}))
\[
\bar{\beta}Ce^{i\phi}\partial_v\bar{\eta}=-\beta Ce^{-i\phi}\partial_v\eta.
\]
Then
\begin{align}
|\hat{M}|=&-\sqrt{t+w^2}(\bar{\beta}Ce^{i\phi}\partial_u\bar{\eta} +\beta Ce^{-i\phi}\partial_u\eta) \nonumber\\
&=\frac{\sqrt{t+w^2}}{\partial_v\bar{\eta}}\beta Ce^{-i\phi}(\partial_v\eta\partial_u\bar{\eta}-\partial_u\eta\partial_v\bar{\eta}).\nonumber
\end{align}
This is well-defined since, by equation (\ref{e:onf2}) we have $\partial_v\eta\neq 0$. 

Now, for the sake of contradiction, suppose that $|\hat{M}|=0$. Then from the above expression and the fact that $\beta\neq 0$ and $C\neq 0$, 
we see that this implies that $\partial_u\eta=\lambda\partial_v\eta$ for some non-zero $\lambda\in{\mathbb R}$. Substituting this in equations (\ref{e:onf1}) and (\ref{e:onf3}) we find that
\begin{equation}\label{e:onf4}
{\textstyle{\frac{1}{2i}}}(\partial_u\xi\partial_v\bar{\eta}-\partial_u\bar{\xi}\partial_v\eta)={\textstyle{\frac{1}{\lambda}}},
\end{equation}
\begin{equation}\label{e:onf5}
(\partial_u\xi+\lambda\partial_v\xi)\partial_v\bar{\eta}-(\partial_u\bar{\xi}+\lambda\partial_v\bar{\xi})\partial_v\eta=0.
\end{equation}
Now adding equation (\ref{e:onf4}) to $\lambda$ times equation (\ref{e:onf5}) yields
\[
{\textstyle{\frac{1}{2i}}}[(\partial_u\xi+\lambda\partial_v\xi)\partial_v\bar{\eta}-(\partial_u\bar{\xi}+\lambda\partial_v\bar{\xi})\partial_v\eta]=
   \lambda+{\textstyle{\frac{1}{\lambda}}}.
\]
Comparing this with (\ref{e:onf5}) we find that $\lambda+\lambda^{-1}=0$ which is impossible. We conclude that $|\hat{M}|\neq0$, and hence the 
Lopatinskii-Shapiro conditions hold. Thus the  boundary conditions (ii) to (iv) of the system {\bf I.B.V.P} satisfy the complementarity condition, and therefore, coupled with Cauchy initial data (i), is strongly parabolic \cite{Eid}. 

The proof of short-time existence now proceeds as in \cite{Freire}, which we now outline. Consider the set $Q_T=\overline{D}\times[0,T]$ and the space 
\[
B_R^T=\{f\in C^{2+\alpha}(Q_T, {\mathbb R}^4)\;|\; f(s=0)=f_0,\; \|f-f_0\|_{1+\alpha}\leq R  \}.
\]
Given $f\in B_R^T$, we can solve the linearized initial boundary value problem in the beginning of the proof to obtain $\hat{f}$, 
which by strict parabolicity and Theorem VI.21 of \cite{Eid} satisfies the following a priori estimate:
\[
\|\hat{f}\|_{C^{2+\alpha}(Q_T)}\leq c\left(\|g\|_{C^{\alpha}(Q_T)}
+\|f_0\|_{C^{\alpha}(D)}+\sum_q \|h^q\|_{C^{2+\alpha-r_q}(\partial D\times[0,T])}   \right),
\]
\noindent
where $g(x,t)$ depends on the induced metric and ambient Christoffel symbols at $f$ literally as in (5.1) of \cite{LaS}, and $h$ on the boundary data at $f$ as they appear in (ii) - (iv) above. 

The  pseudo-linearization of the system at $f_0$, for fixed ${f}$ is
\[
\frac{\partial \hat{f}}{\partial s}-\triangle_{df_0}\hat{f}=\triangle_{d{f}}{f}-\triangle_{df_0}{f},
\]
where $\triangle_{d\tilde{f}}$ is the mean curvature operator with  symbol linearized at $\tilde{f}$. 

Note that a fixed point of the map ${\mathcal G}$ taking ${f}$ to a solution $\hat{f}$ of the pseudo-linearized equations is a solution of the original quasilinear system.

\begin{Lem}\cite{Freire}
There exists an $R$ and $T$ depending on $f_0$ such that ${\mathcal G}: B_R^T\rightarrow B_R^T$.
\end{Lem}
\begin{proof} 
By the above regularity estimate, we have a uniform $C^{2+\alpha}$ - bound for every $ {\mathcal G}({f})$ for all ${f} \in B^T_R $.  As a result, we may bound the $C^{1+\alpha}$ - norm of $\hat{f}$ by a $c_1 T^\epsilon$, for some $\epsilon>0$. Therefore, choosing $T$ sufficiently small, the claim follows.
\end{proof}

Short-time existence now follows from an application of the Schauder fixed point theorem, since, $B_R^T$ is compact and convex in the $C^{1+\alpha} $ - norm, and for sufficiently small $T$, ${\mathcal G}$ is a contraction.
\end{proof}

\vspace{0.1in}

\subsection{Long-time existence}\label{ss:lte}

Let ${\mathcal P}$ be a totally real Lagrangian section of $TS^2$ that projects to a hemisphere and $\gamma_0\in {\mathcal P}$ projects to the pole. By an isometry of $TS^2$ we can choose $\gamma_0$ to have Gauss coordinate $\xi=0$. Assume ${\mathcal P}$ is totally real, that is, it is nowhere holomorphic: $|\sigma|\neq0$.

Let $F_0$ be the graph associated with ${\mathcal P}$ and add a holomorphic twist (as in Proposition \ref{p:linholtwist}):
\[
\tilde{F}=F_0-iC_0\xi,
\]
to form the section $\tilde{\mathcal P}$. While the aholomorphicity remains unchanged ($|\tilde{\sigma}|=|\sigma|$), $\tilde{\mathcal P}$ is not Lagrangian in the open unit hemisphere, since
\[
\tilde{\lambda}=-C_0\frac{1-R^2}{1+R^2},
\]
where $R=|\xi|$. Fix $C_0>|\tilde{\sigma}(0)|$ and, as before, denote the set on which the induced metric is positive by
\[
\tilde{\Lambda}_{C_0}'=\left\{\;\gamma\in\tilde{\mathcal P}\;\;\left|\;\; |\tilde{\sigma}(\gamma)|<|\tilde{\lambda}(\gamma)|\;\right.\right\}.
\] 
Clearly $\tilde{\Lambda}_{C_0}'$ is non-empty since it contains $0$. Denote the connected component of $0$ in $\tilde{\Lambda}_{C_0}'$ by $\tilde{\Lambda}_{C_0}$. Thus, $\tilde{\Lambda}_{C_0}$ is a positive section over an open subset in the unit hemisphere.

\begin{center}
\includegraphics[width=98mm]{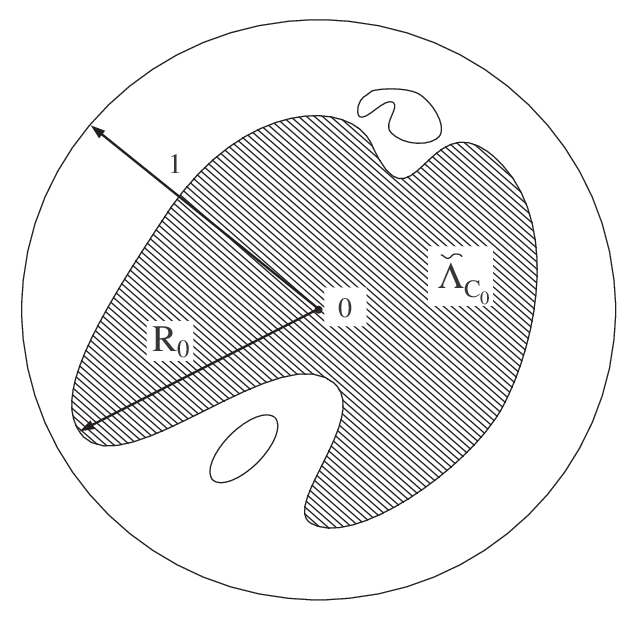}
\end{center}

In what follows we take the sup and inf of $|\tilde{\sigma}|$ over the hemisphere, while for other quantities, such as $|\tilde{\lambda}|$, sup and inf will be over the positive region $\tilde{\Lambda}_{C_0}$. Denote the maximum Gauss radius of the positive region by
\[
R_0=\sup_{\tilde{\Lambda}_{C_0}} |\xi|.
\]
The following diagram illustrates these definitions.

Note that $0<R_0<1$ and 
\[
\sup_{\tilde{\Lambda}_{C_0}}|\tilde{\lambda}|=C_0,
\qquad\qquad
\inf_{\tilde{\Lambda}_{C_0}}|\tilde{\lambda}|=C_0\frac{1-R_0^2}{1+R_0^2}.
\]

In this section we prove the following:

\begin{Thm}\label{t:lte}
Let ${\mathcal P}$,  $C_0$, $\tilde{\mathcal P}$ and $R_0$ be as above. Then there exist constants $B_1$ and $C_1$ s.t. $\forall B,C$ satisfying
\[
0<B<B_1(C_0,R_0,\sup|\tilde{\sigma}|,\inf|\tilde{\sigma}|)
\qquad\qquad
0<C<C_1(C_0,R_0,\sup|\tilde{\sigma}|,\inf|\tilde{\sigma}|),
\]
the {\bf I.B.V.P.} with initial constants $B$ and $C$ has a solution for all time.
\end{Thm}
\begin{proof}
In order to prove long-time existence for the {\bf I.B.V.P.} we show that it is uniformly parabolic and, since it is quasilinear,
we need a global gradient estimate. 

In Theorem 1 of \cite{Gk02}, long-time existence for compact flowing submanifolds is proved, so long as the ambient space satisfies the timelike curvature condition and the submanifold stays in a compact region of the ambient space. The proof was based on the maximum principle to bound the gradient and the mean curvature. These arguments extend to the case where these functions have an interior maximum and so, to prove this theorem, we extend the estimates to the edge.

More specifically, we ensure that the timelike curvature condition holds for the flow and that the surface stays in a compact region of $TS^2$. This we do in Propositions \ref{p:tcc} and \ref{p:compact} below, respectively.

We also find estimates for the gradient and mean curvature vector at the edge. The first of these is established by ensuring that the edge stays in a positive region of the boundary plane (Proposition \ref{p:posregion}) and the first derivative remains bounded (Proposition \ref{p:bdryest1}). To establish the latter, we first control the angles which the edge makes with the canonical orthonormal frame (Proposition \ref{p:anglecontrol}) and then extract bounds on the 2-jet of the flowing surface in terms of the 2-jet of the boundary plane (Theorem \ref{t:bdryest3}). Thus we bound $|H|$ and have established the gradient estimate (see Proposition 7 of \cite{Gk02}).

Once we have the global gradient estimate, the proof proceeds as follows. By the short-time existence Theorem \ref{t:ste} we know that there exists an $s_0$ such that there is a solution to initial boundary value problem with
\[
f\in C^{2+\alpha}(\overline{D}\times[0,s_0)). 
\]
By a modification of the interior regularity results in \cite{thorpe} \cite{white}, the gradient estimate implies that this solution can be extended to
\[
f\in C^{2+\alpha}(\overline D\times[0,s_0]). 
\]
In particular, as in the case of definite metrics, interior regularity can be established by the isometric embedding of the ambient space in a flat pseudo-Euclidean space of higher dimension. 

The existence of embeddings in which only the timelike dimension is increased has long been known - for example a neutral 4-manifold can always be locally embedded in ${\mathbb R}^{2+8}$ according to Theorem 1 of \cite{Friedman}, and then the pseudo-Riemannian heat kernel and monotonicity formula introduced in \cite{thorpe} can be used to establish the stated regularity.

Let $s_{max}$ be the maximum time of existence for the initial boundary problem:
\[ 
s_{max}={\mbox{inf}}_{U\subset\subset D}\{ {\mbox{sup}}\;s \in{\mathbb R}_{\geq0}\;|\;  f\in C^{2+\alpha}(U\times[0,s)) \}.
\]
Assume for the sake of contradiction that $s_{max}<\infty$. That is, there exists $U\subset\subset D$ and a sequence of times 
$\{s_j\}\rightarrow s_{max}$ such that 
\[
\lim_{j\rightarrow\infty}\|f_{s_j}\|_{C^{2+\alpha}(U)}\rightarrow\infty.
\]
Given the a priori gradient estimate, this contradicts the regularity result and so we have that $s_{max}$ is 
infinite.

\end{proof}

\subsubsection{Interior gradient estimate}\label{ss:ige}

We start by showing that the timelike curvature condition holds along the flow and that the flow stays in a compact region. First, recall Definition 1 of \cite{Gk02}:

\vspace{0.1in}
\begin{Def}
The manifold $({\mathbb M},{\mathbb G})$ is said to satisfy the {\it timelike curvature condition} if, for any spacelike $n$-plane ${\mathcal P}$ at a point in ${\mathbb M}$, the Riemann curvature tensor satisfies
\begin{equation}\label{e:tcc}
{\mathbb G}(\overline{R}(X,\tau_i)X,\tau_i)\;\geq k \;{\mathbb G}(X,X),
\end{equation}
for some positive constant $k$, where $\{\tau_i\}_{i=1}^n$ form an orthonormal basis for ${\mathcal P}$ and $X$ is any timelike vector orthogonal to ${\mathcal P}$.
\end{Def}
\vspace{0.1in}

We now prove:
\vspace{0.1in}
\begin{Prop}\label{p:tcc}
The timelike curvature condition (\ref{e:tcc}) holds when ${\mathcal P}$ is the tangent plane of the flowing surface, for as long as the flow exists.
\end{Prop}
\begin{proof}
Given a timelike plane ${\mathcal P}$, chose an adapted orthonormal frame $E_{(\mu)}=\{e_{1)},e_{(2)},f_{(1)},f_{(2)}\}$ and for a timelike vector $X=X^{(1)}f_{(1)}+X^{(2)}f_{(2)}$ compute
\[
<\overline{R}(X,e_{(a)})X,e_{(a)}>=\frac{2|\sigma|}{\lambda^2-|\sigma|^2}(-(X^{(1)})^2+(X^{(2)})^2)\geq\frac{2|\sigma|}{\lambda^2-|\sigma|^2}{\mathbb G}(X,X).
\]
From Proposition \ref{p:shearest} we have that
\[
\left(\frac{\partial }{\partial s}-{\mathbb G}^{jk}\partial_j\partial_k\right)\left(\frac{|\sigma|^2}{\lambda^2-|\sigma|^2}\right)\leq0,
\]
and so we have the interior a priori bound
\[
\frac{|\sigma|^2}{\lambda^2-|\sigma|^2}\leq C_1.
\]
In addition, we claim that under the flow
\[
\Delta=\lambda^2-|\sigma|^2\geq C_2.
\]
To see this, note that at the edge this follows from Corollary \ref{c:bdryests}. On the other hand at an interior minimum of $\Delta$ we have that $\partial\Delta=2\lambda\partial\lambda-\sigma\partial\bar{\sigma}-\bar{\sigma}\partial\sigma=0$.
Substituting this in the flow for $\Delta$, as expressed in Corollary \ref{c:flowdet}, we find that
\begin{align}
\left(\frac{\partial }{\partial s}-{\mathbb G}^{jk}\partial_j\partial_k\right)\Delta&=\frac{2\Delta}{\lambda^2}\|da\|^2
    +\frac{2ai(1+\xi\bar{\xi})^2}{\lambda}[\bar{\partial}a(\partial\phi-{\textstyle{\frac{2i\bar{\xi}}{1+\xi\bar{\xi}}}})
          -\partial a(\bar{\partial}\phi+{\textstyle{\frac{2i\xi}{1+\xi\bar{\xi}}}})]\nonumber\\
&\qquad\qquad\qquad\qquad+a^2\|d\phi-2j\;d\ln(1+\xi\bar{\xi})\|^2-4\lambda,
\end{align}
where we have introduced $\sigma=ae^{i\phi}$ and $\|da\|^2={\mathbb{G}}^{jk}\partial_ja\;\partial_ka$ is the square norm of the gradient.

Now, to identify the second term on the right-hand side
\[
i[\bar{\partial}a(\partial\phi-{\textstyle{\frac{2i\bar{\xi}}{1+\xi\bar{\xi}}}})-\partial a(\bar{\partial}\phi+{\textstyle{\frac{2i\xi}{1+\xi\bar{\xi}}}})]d\xi\wedge d\bar{\xi}=(d\phi-2j\;d\ln(1+\xi\bar{\xi}))\wedge da,
\]
and taking the Hodge star operator with respect to the  metric ${\mathbb{G}}$ we get
\[
\frac{i(1+\xi\bar{\xi})^2}{2\sqrt{\Delta}}[\bar{\partial}a(\partial\phi-{\textstyle{\frac{2i\bar{\xi}}{1+\xi\bar{\xi}}}})-\partial a(\bar{\partial}\phi+{\textstyle{\frac{2i\xi}{1+\xi\bar{\xi}}}})]=\star[(d\phi-2j\;d\ln(1+\xi\bar{\xi}))\wedge da].
\]
By elementary geometry we have that for any 1-forms $\alpha,\beta$
\[
\star[\alpha\wedge \beta]\geq-\|\alpha\|.\|\beta\|,
\]
and applying this, we find
that
\begin{align}
\left(\frac{\partial }{\partial s}-{\mathbb G}^{jk}\partial_j\partial_k\right)\Delta&\geq\frac{\Delta}{\lambda^2}\|da\|^2
-\frac{2a\sqrt{\Delta}}{\lambda}\|d\phi-2j\;d\ln(1+\xi\bar{\xi})\|.\|da\|\nonumber\\
&\qquad\qquad\qquad\qquad   +a^2\|d\phi-2j\;d\ln(1+\xi\bar{\xi})\|^2-4\lambda\nonumber\\
&=\left(\frac{\sqrt{\Delta}}{\lambda}\|da\|-a\|d\phi-2j\;d\ln(1+\xi\bar{\xi})\|\right)^2-4\lambda\geq0.\nonumber
\end{align}
The interior bound on $\Delta$ follows.

Therefore
\[
\left(\frac{|\sigma|}{\lambda^2-|\sigma|^2}\right)^2=\frac{|\sigma|^2}{\lambda^2-|\sigma|^2}\frac{1}{\lambda^2-|\sigma|^2}\leq \frac{C_1}{C_2}.
\]
Thus
\[
{\mathbb G}(\overline{R}(X,e_{(a)})X,e_{(a)}))\;\geq k|X|^2,
\]
as claimed.
\end{proof}

\vspace{0.1in}

Turn now to showing that the flow stays in a compact region. We do so by showing that the graph function is bounded in the fibre directions of $TS^2$ by bounding the perpendicular distance function $\chi$ (see Definition \ref{d:perpdist}).

\begin{Prop}\label{p:compact}
Under mean curvature flow the perpendicular distance function satisfies
\[
\chi\leq C_1(\chi_0,\tilde{\chi}),
\]
where $\chi_0$ and $\tilde{\chi}$ are the initial and boundary perpendicular distances to the origin.
\end{Prop}

\begin{proof}
At the edge $f(\partial D)$ the Dirichlet condition means that $F_s=\tilde{F}$ so that
\[
\chi(\partial D)=\tilde{\chi}(\partial D)\leq C_1.
\]
Thus we need only consider the interior.

First note that
\[
\partial\chi^2=4\bar{F}\partial\left(\frac{F}{(1+\xi\bar{\xi})^2}\right)+\frac{4F}{(1+\xi\bar{\xi})^2}\partial\bar{F}
    =4\frac{\bar{F}\rho-F\sigma}{(1+\xi\bar{\xi})^2},
\]
and so the expression in Proposition \ref{p:chiflow} can be rewritten
\[
\left(\frac{\partial }{\partial s}-{\mathbb G}^{jk}\partial_j\partial_k\right)\chi^2=\frac{i(F\partial\chi^2-\bar{F}\bar{\partial}\chi^2)
   -i(\sigma\xi\bar{\partial}\chi^2-\bar{\sigma}\bar{\xi}\partial\chi^2)(1+\xi\bar{\xi})+4\lambda(\rho\bar{\rho}-\sigma\bar{\sigma})}
  {\lambda^2-\sigma\bar{\sigma}}.
\]
Thus, at an interior maximum of $\chi$, $\partial\chi=\bar{\partial}\chi=0$ and so
\[
\left(\frac{\partial }{\partial s}-{\mathbb G}^{jk}\partial_j\partial_k\right)\chi^2
=\frac{4\lambda(\rho\bar{\rho}-\sigma\bar{\sigma})}{\lambda^2-\sigma\bar{\sigma}}\leq0.
\]
By the maximum principle, the result follows.

\end{proof}

\vspace{0.1in}

Having established these two propositions, we move on to the edge estimates.

\vspace{0.1in}

\subsubsection{Boundary estimates}\label{ss:best}

In this section we establish two things: a gradient estimate and an estimate on the norm of the mean curvature vector at the edge of the flowing disc.

To start, we show that the flowing disc cannot become degenerate at the edge by becoming tangent to a fibre of $TS^2\rightarrow S^2$.

\begin{Prop}\label{p:bdryest1}
At the edge, $|dF|^2<C_1(|d\tilde{F}|^2)$.
\end{Prop}
\begin{proof}
To bound the gradient of the graph function $F(\xi,\bar{\xi})$ we need to bound the slopes $\rho=\vartheta+i\lambda$ and $\sigma$. The asymptotic holomorphicity condition at the edge says that
\[
|\sigma|=\frac{C}{1+s}\leq C,
\]
and so we have a bound on $|\sigma|$. On the other hand the intersection condition implies that
\[
\rho=\tilde{\rho}+(\bar{\sigma}-\bar{\tilde{\sigma}})e^{i\beta},
\]
for some $\beta$, and so $|\rho|\leq |\tilde{\rho}|+|\sigma|+|\tilde{\sigma}|<C$, since the gradient of the graph function $\tilde{F}(\xi,\bar{\xi})$ of the boundary plane is bounded.
\end{proof}

\vspace{0.1in}

To complete the gradient estimate at the edge we must now ensure that the flowing surface does not pick up a degenerate direction which is not tangent to the fibre of $TS^2$. That is, we must ensure that $\lambda^2-|\sigma|^2>0$ at the edge. 

To ensure this we note that, by Corollary \ref{c:nulltogether}, $\lambda^2-|\sigma|^2=0$ can only happen if the edge of the flowing disc hits a degenerate point on the boundary plane $\tilde{\mathcal P}$,  i.e. if $\tilde{\lambda}^2-|\tilde{\sigma}|^2=0$. Thus, if we can constrain the flowing edge to lie in a positive region  $\tilde{\Lambda}_{C_0}\subset \tilde{\mathcal P}$, we will have the required gradient estimate.

\begin{Prop}\label{p:posregion}
Suppose 
\begin{equation}\label{e:smallness1}
C<\min\left\{\frac{\inf|\tilde{\sigma}|}{(1+4\cosh B)^2},\frac{\inf|\tilde{\sigma}|}{\sup|\tilde{\sigma}|}\right\}.
\end{equation}
If $f_s(\partial D)\subset\tilde{\Lambda}_{C_0}$ for $s=0$, then $f_s(\partial D)\subset\tilde{\Lambda}_{C_0}$ for all $s$.
\end{Prop}
\begin{proof}
Assume for the sake of contradiction that there exists a point and time $(\gamma_0,s_0)$ such that $\tilde{\mu}(\gamma_0,s_0)=1$. Then, by 
Corollary \ref{c:nulltogether}, $\mu(\gamma_0,s_0)=1$. Substituting these into equation (\ref{e:hypang2}) we find that at that point
\begin{align}
\cosh B&=\frac{\left|\frac{\sigma}{\tilde{\sigma}}\right|^{\scriptstyle{\frac{1}{2}}}
   +\left|\frac{\tilde{\sigma}}{\sigma}\right|^{\scriptstyle{\frac{1}{2}}}-\sin 2\theta\sin 2\tilde{\theta}}
{(1+\cos 2\theta)(1+\cos 2\tilde{\theta})}\geq {\textstyle{\frac{1}{4}}}\left[\left|\frac{\sigma}{\tilde{\sigma}}\right|^{\scriptstyle{\frac{1}{2}}}
   +\left|\frac{\tilde{\sigma}}{\sigma}\right|^{\scriptstyle{\frac{1}{2}}}-1\right]\nonumber\\
&\geq {\textstyle{\frac{1}{4}}}\left[\left(\frac{|\sigma|}{\sup|\tilde{\sigma}|}\right)^{\scriptstyle{\frac{1}{2}}}
   +\left(\frac{\inf|\tilde{\sigma}|}{|\sigma|}\right)^{\scriptstyle{\frac{1}{2}}}-1\right]> {\textstyle{\frac{1}{4}}}\left[\left(\frac{\inf|\tilde{\sigma}|}{C}\right)^{\scriptstyle{\frac{1}{2}}}-1\right],\nonumber
\end{align}
where the final inequality follows from the assumption
\[
|\sigma|\leq C<\frac{\inf|\tilde{\sigma}|}{\sup|\tilde{\sigma}|}.
\]
Rearranging the inequality, we find that at $(\gamma_0,s_0)$ 
\begin{equation}\label{e:eqq1}
1+4\cosh B-\frac{\inf|\tilde{\sigma}|^{\scriptstyle{\frac{1}{2}}}}{C^{\scriptstyle{\frac{1}{2}}}}>0.
\end{equation}
But, by assumption
\[
C<\frac{\inf|\tilde{\sigma}|}{(1+4\cosh B)^2},
\]
and so
\[
1+4\cosh B-\frac{\inf|\tilde{\sigma}|^{\scriptstyle{\frac{1}{2}}}}{C^{\scriptstyle{\frac{1}{2}}}}<0,
\]
which contradicts (\ref{e:eqq1}). Thus the flow never reaches such a point $(\gamma_0,s_0)$.
\end{proof}
\vspace{0.1in}

\begin{Cor}\label{c:bdryests}
Propositions \ref{p:bdryest1} and \ref{p:posregion} imply that if the initial aholomorphicity is chosen to satisfy inequality (\ref{e:smallness1}), there exist positive constants $C_1,...,C_7$ depending only on boundary and initial quantities such that, for as long as the flow exists, we have at the edge
\[
C_1<\lambda^2-|\sigma|^2<C_2, \qquad 0\leq\mu<C_3<1,
\]
\[
C_4<\tilde{\lambda}^2-|\tilde{\sigma}|^2<C_5, \qquad C_6<\tilde{\mu}<C_7<1.
\]
\end{Cor}
\vspace{0.1in}

We finally turn to establishing a bound on the norm of the mean curvature at the edge. This is necessary because the interior gradient estimate from \cite{Gk02} depends on a bound on $|H|$ (see Proposition 7 of \cite{Gk02}), and while we have an interior estimate for $|H|$ thanks to Proposition 9 of \cite{Gk02}, we still need to show that $|H|$ does not blow-up at the edge. In fact, we will bound all of the second derivatives.

To this end, working at a point on $f_s(\partial D)\subset\tilde{\Lambda}_{C_0}$, the tangent space of the ambient manifold splits in two different ways:
\[
TTS^2=TD\oplus ND=T\tilde{\Lambda}_{C_0}\oplus N\tilde{\Lambda}_{C_0},
\]
where we drop the $f_s$ on the flowing disc for ease of notation here and in what follows. Let $\{\mathring{e}_{(a)},\mathring{f}_{(b)}\}$ 
be an adapted orthonormal frame for the first splitting with co-frame $\{\mathring{e}^{(a)},\mathring{f}^{(b)}\}$ and similarly define 
$\{\mathring{\tilde{e}}_{(a)},\mathring{\tilde{f}}_{(b)}\}$ and  
$\{\mathring{\tilde{e}}^{(a)},\mathring{\tilde{f}}^{(b)}\}$ for the second splitting. 

Chose these frames so that the first tangent vectors align with the common intersection of the tangent planes, while the second normal vectors align with the common intersection of the normal planes. From Corollary \ref{c:multi} this means that, if $B$ is the hyperbolic angle between the planes, the following relations hold:
\[
\mathring{\tilde{e}}^{(1)}=\mathring{e}^{(1)},
\qquad
\mathring{\tilde{e}}^{(2)}=\cosh B\;\mathring{e}^{(2)}+\sinh B \;\mathring{f}^{(1)},
\]
\[
\mathring{\tilde{f}}^{(1)}=\sinh B\;\mathring{e}^{(2)}+\cosh B \;\mathring{f}^{(1)},
\qquad
\mathring{\tilde{f}}^{(2)}=\mathring{f}^{(2)},
\]
and
\[
\mathring{e}_{(1)}=\mathring{\tilde{e}}_{(1)},
\qquad
\mathring{e}_{(2)}=\cosh B\;\mathring{\tilde{e}}_{(2)}+\sinh B \;\mathring{\tilde{f}}_{(1)},
\]
\[
\mathring{f}_{(1)}=\sinh B\;\mathring{\tilde{e}}_{(2)}+\cosh B \;\mathring{\tilde{f}}_{(1)},
\qquad
\mathring{f}_{(2)}=\mathring{\tilde{f}}_{(2)}.
\]

Suppose that these aligned frames are related to the canonical frames in Proposition \ref{p:frames} by angles  
$\theta,\tilde{\theta},\psi,\tilde{\psi} $: 

\[
\mathring{e}_{(1)}=\cos\theta {e}_{(1)}-\sin\theta {e}_{(2)},
\qquad\qquad
\mathring{e}_{(2)}=\sin\theta {e}_{(1)}+\cos\theta {e}_{(2)},
\] 
\[
\mathring{f}_{(1)}=\cos\psi {f}_{(1)}+\sin\psi {f}_{(2)},
\qquad\qquad
\mathring{f}_{(2)}=-\sin\psi {f}_{(1)}+\cos\psi {f}_{(2)},
\] 
\[
\mathring{\tilde{e}}_{(1)}=\cos\tilde{\theta} {\tilde{e}}_{(1)}-\sin\tilde{\theta} {\tilde{e}}_{(2)},
\qquad\qquad
\mathring{\tilde{e}}_{(2)}=\sin\tilde{\theta} {\tilde{e}}_{(1)}+\cos\tilde{\theta} {\tilde{e}}_{(2)},
\]
\[
\mathring{\tilde{f}}_{(1)}=\cos\tilde{\psi} {\tilde{f}}_{(1)}+\sin\tilde{\psi} {\tilde{f}}_{(2)},
\qquad\qquad
\mathring{\tilde{f}}_{(2)}=-\sin\tilde{\psi} {\tilde{f}}_{(1)}+\cos\tilde{\psi} {\tilde{f}}_{(2)}.
\]

We now set out to control the values of these angles under the flow. First
\begin{Prop}
If we chose $B$ initially satisfying
\begin{equation}\label{e:smallness2}
B<\tanh^{-1}\left(\frac{\inf|\tilde{\sigma}|}{C_0}\right),
\end{equation}
then $\tilde{\mu}>\tanh B$ at $f_s(\partial D)$ for all $s$.
\end{Prop}
\begin{proof}
This follows simply from
\[
\tilde{\mu}=\frac{|\tilde{\sigma}|}{|\tilde{\lambda}|}\geq\frac{\inf|\tilde{\sigma}|}{C_0}>\tanh B
\]
\end{proof}

\begin{Prop}\label{p:anglecontrol}
If, in addition to the inequalities (\ref{e:smallness1}) and (\ref{e:smallness2}), the following holds:
\begin{equation}\label{e:smallness3}
C<C_0\frac{1-R_0^2}{1+R_0^2}\min\left\{1-\cosh B \left(1-{\textstyle{\frac{\inf|\tilde{\sigma}|^2}{C_0^2}}}\right)^{\scriptstyle{\frac{1}{2}}},
\cosh B \left(1-{\textstyle{\frac{\sup|\tilde{\sigma}|^2(1+R_0^2)^2}{C_0^2(1-R_0^2)^2}}}\right)^{\scriptstyle{\frac{1}{2}}}\right\},
\end{equation}
and if $0<\theta<{\textstyle{\frac{\pi}{2}}}$ and  $0<\psi<{\textstyle{\frac{\pi}{2}}}$ initially, then these bounds on the angles hold for as long as the flow exists.
\end{Prop}
\begin{proof}
Suppose for the sake of contradiction that at some point and time $(\gamma_0,s_0)$ we have $\theta={\textstyle{\frac{\pi}{2}}}$. Note that the following argument holds even if $s_0=\infty$. The expression of the hyperbolic angle (\ref{e:hypangle}) and equation (\ref{e:intersectionconstraint}) means that at $(\gamma_0,s_0)$,
\[
\cosh B=\frac{(1+\mu)\left(1-{\textstyle{\frac{|\sigma|}{|\tilde{\lambda}|}}}\right)}
{(1-\mu^2)^{\scriptstyle{\frac{1}{2}}}(1-\tilde{\mu}^2)^{\scriptstyle{\frac{1}{2}}}},
\]
or, after rearrangement
\[
\mu=\frac{\cosh^2B(1-\tilde{\mu}^2)-\left(1-{\textstyle{\frac{|\sigma|}{|\tilde{\lambda}|}}}\right)^2}
{\left(1-{\textstyle{\frac{|\sigma|}{|\tilde{\lambda}|}}}\right)^2+\cosh^2B(1-\tilde{\mu}^2)}.
\]
But
\begin{align}
\cosh B(1-\tilde{\mu}^2)^{\scriptstyle{\frac{1}{2}}}-1+{\textstyle{\frac{|\sigma|}{|\tilde{\lambda}|}}}
&\leq\cosh B(1-\tilde{\mu}^2)^{\scriptstyle{\frac{1}{2}}}-1+{\textstyle{\frac{C}{\inf|\tilde{\lambda}|}}}\nonumber\\
&\leq\cosh B\left(1-{\textstyle{\frac{\inf|\tilde{\sigma}|^2}{C_0^2}}}\right)^{\scriptstyle{\frac{1}{2}}}-1
+{\textstyle{\frac{C(1+R_0^2)}{C_0(1-R_0^2)}}}<0,\nonumber
\end{align}
by our assumption (\ref{e:smallness3}) and so $\mu<0$. But this contradicts $\mu\geq0$ and so $\theta$ is never equal to ${\textstyle{\frac{\pi}{2}}}$.

On the other hand, suppose for the sake of contradiction that at some point and time $(\gamma_0,s_0)$ we have $\theta=0$. In this case the hyperbolic angle turns out to be
\[
\cosh B=\frac{(1-\mu)\left(1+{\textstyle{\frac{|\sigma|}{|\tilde{\lambda}|}}}\right)}{(1-\mu^2)^{\scriptstyle{\frac{1}{2}}}(1-\tilde{\mu}^2)^{\scriptstyle{\frac{1}{2}}}},
\]
which means that at $(\gamma_0,s_0)$ 
\[
\mu=\frac{\left(1-{\textstyle{\frac{|\sigma|}{|\tilde{\lambda}|}}}\right)^2-\cosh^2B(1-\tilde{\mu}^2)}{\left(1-{\textstyle{\frac{|\sigma|}{|\tilde{\lambda}|}}}\right)^2+\cosh^2B(1-\tilde{\mu}^2)}.
\]
Putting this into the left-hand inequality of (\ref{e:ineq2}) we find
\begin{align}
0&\leq \mu-\frac{\tilde{\mu}-\cosh B\sinh B\;(1-\tilde{\mu}^2)}{\tilde{\mu}^2+\cosh^2 B\;(1-\tilde{\mu}^2)}\nonumber\\
&=\frac{\left(1-{\textstyle{\frac{|\sigma|}{|\tilde{\lambda}|}}}\right)^2-\cosh^2B(1-\tilde{\mu}^2)}{\left(1-{\textstyle{\frac{|\sigma|}{|\tilde{\lambda}|}}}\right)^2+\cosh^2B(1-\tilde{\mu}^2)}-\frac{\tilde{\mu}-\cosh B\sinh B\;(1-\tilde{\mu}^2)}{\tilde{\mu}^2+\cosh^2 B\;(1-\tilde{\mu}^2)}\nonumber\\
&\leq\frac{\left[\left(1-{\textstyle{\frac{|\sigma|}{|\tilde{\lambda}|}}}\right)^2-\cosh^2B(1-\tilde{\mu}^2)\right](\tilde{\mu}^2+\cosh^2 B\;(1-\tilde{\mu}^2))}{(1+\cosh^2B(1-\tilde{\mu}^2))^2}\nonumber\\
&\qquad\qquad\qquad-\frac{(\tilde{\mu}-\cosh B\sinh B\;(1-\tilde{\mu}^2))\left[\left(1-{\textstyle{\frac{|\sigma|}{|\tilde{\lambda}|}}}\right)^2+\cosh^2B(1-\tilde{\mu}^2)\right]}{(1+\cosh^2B(1-\tilde{\mu}^2))^2}\nonumber\\
&=\frac{\left(1-{\textstyle{\frac{|\sigma|}{|\tilde{\lambda}|}}}\right)^2(\tilde{\mu}^2+\cosh^2 B\;(1-\tilde{\mu}^2)-\tilde{\mu}+\cosh B\sinh B\;(1-\tilde{\mu}^2))}{(1+\cosh^2B(1-\tilde{\mu}^2))^2}\nonumber\\
&\qquad\qquad\qquad-\frac{\cosh^2B(1-\tilde{\mu}^2)(\tilde{\mu}^2+\cosh^2 B\;(1-\tilde{\mu}^2)+\tilde{\mu}-\cosh B\sinh B\;(1-\tilde{\mu}^2))}{(1+\cosh^2B(1-\tilde{\mu}^2))^2},\label{e:ineq0}
\end{align}
where the middle inequality follows from using $1-|\sigma|/|\tilde{\lambda}|\leq1$ and $\tilde{\mu}\leq 1$.

Noting that for $0<\tilde{\mu}<1$ we have $\tilde{\mu}^2+\cosh^2 B\;(1-\tilde{\mu}^2)-\tilde{\mu}+\cosh B\sinh B\;(1-\tilde{\mu}^2)>0$ and that $C=|\sigma|(1+s)$ along the boundary, we can rearrange inequality (\ref{e:ineq0}) and conclude that at $(\gamma_0,s_0)$ at the  
\begin{align}
C&\geq|\sigma|\geq|\tilde{\lambda}|\cosh B(1-\tilde{\mu}^2)^{\scriptstyle{\frac{1}{2}}}\left[\frac{\tilde{\mu}^2+\cosh^2 B\;(1-\tilde{\mu}^2)+\tilde{\mu}-\cosh B\sinh B\;(1-\tilde{\mu}^2)}{\tilde{\mu}^2+\cosh^2 B\;(1-\tilde{\mu}^2)-\tilde{\mu}+\cosh B\sinh B\;(1-\tilde{\mu}^2)}\right]^{\scriptstyle{\frac{1}{2}}}\nonumber\\
&\geq|\tilde{\lambda}|\cosh B(1-\tilde{\mu}^2)^{\scriptstyle{\frac{1}{2}}}\nonumber\\
&\geq C_0\frac{1-R_0^2}{1+R_0^2}\cosh B \left(1-{\textstyle{\frac{\sup|\tilde{\sigma}|^2(1+R_0)^2}{C_0^2(1-R_0)^2}}}\right)^{\scriptstyle{\frac{1}{2}}}.\nonumber
\end{align}
But this contradicts assumption (\ref{e:smallness3}). Thus $\theta$ is never 0.

These exact same arguments go through upon replacing $\theta$ by $\psi$, since all of the relevant equations also hold for $\psi$ (see the proof of Proposition \ref{p:hypang1}). Thus $\psi\neq 0,{\textstyle{\frac{\pi}{2}}}$.
\end{proof}

\vspace{0.1in}

Having established control of the angles between the intersection and the canonical frames we now set about bounding the 2-jet of the flowing surface.  As the boundary is moving we return to the original version of the flow in the {\bf I.B.V.P.} without the perpendicular projection on the time derivative. Note also that, since the mean curvature vector is perpendicular to the flowing surface, it is perpendicular to the intersection of the flowing disc and boundary plane and we have in the adapted frames that
\begin{equation}\label{e:compat}
H^{(a)}\mathring{f}_{(a)}=H^{(1)}\mathring{f}_{(1)}=H\mathring {f}_{(1)},
\end{equation}
where, for brevity, we have dropped the frame index so that $H$ is now just a function.

\begin{Prop}
The time and space derivatives of the intersection and constant angle conditions yield the following relationships between the flowing and boundary planes along the edge:
\begin{equation}\label{e:dbc1}
\nabla_{(1)}H=H(\sinh B\;\tilde{\Gamma}_{(112)}+\cosh B\;\tilde{A}^{(\hat{1})}_{(11)}),
\qquad
\nabla_{(2)}H=H\tilde{A}_{(12)}^{(\hat{1})},
\end{equation}
\begin{equation}\label{e:dbc2}
C_{(\hat{1}\hat{2})}^{(1)}=-\tilde{A}^{(\hat{2})}_{(11)},
\qquad
C_{(\hat{1}\hat{2})}^{(2)}=\sinh B\; \tilde{C}_{(\hat{1}\hat{2})}^{(1)}-\cosh B\;\tilde{A}^{(\hat{2})}_{(12)},
\end{equation}
\begin{equation}\label{e:dbc3}
A^{(\hat{1})}_{(12)}=\tilde{A}^{(\hat{1})}_{(12)},
\qquad
A^{(\hat{2})}_{(11)}=\tilde{A}^{(\hat{2})}_{(11)},
\end{equation}
\begin{equation}\label{e:dbc4}
\Gamma_{(112)}=-\cosh B\;\tilde{\Gamma}_{(112)}+\sinh B\; \tilde{A}^{(\hat{1})}_{(11)},
\end{equation}
\begin{equation}\label{e:dbc5}
A^{(\hat{1})}_{(11)}=\sinh B\;\tilde{\Gamma}_{(112)}+\cosh B\; \tilde{A}^{(\hat{1})}_{(11)},
\end{equation}
\begin{equation}\label{e:dbc6}
A^{(\hat{2})}_{(12)}=\cosh B\;\tilde{A}^{(\hat{2})}_{(12)}+\sinh B\; \tilde{C}^{\;\;(\hat{2})}_{(1\hat{1})},
\end{equation}
\begin{equation}\label{e:dbc7}
C^{\;\;(\hat{2})}_{(1\hat{1})}=-\sinh B\;\tilde{A}^{(\hat{2})}_{(12)}-\cosh B\; \tilde{C}^{\;\;(\hat{2})}_{(1\hat{1})}.
\end{equation}
Here and throughout, $A$, $\Gamma$ and $C$ are the second fundamental form, tangent and normal connections of the flowing disc, in components of frames adapted to the intersection, so for example, $A_{(ab)}^{(\hat{c})}=<A(\mathring{e}_{(a)},\mathring{e}_{(b)}),\mathring{f}^{(c)}>$, where a hat on a frame
index means it is in the normal direction. Similar notation holds for tilded quantities.

\end{Prop}
\begin{proof}
In order to differentiate the constant angle condition along the edge, we need to compute the evolution of the adapted frames. Starting with the frame adapted to the boundary plane we compute in  normal coordinates:
\begin{align}
\frac{\partial}{\partial s}\mathring{\tilde{e}}_{(b)}&=H^{(d)}\partial_{(d)}\mathring{\tilde{e}}_{(b)}
  =H^{(d)}\overline{\nabla}_{(d)}\mathring{\tilde{e}}_{(b)}
= H^{(d)}[\tilde{\nabla}_{(d)}\mathring{\tilde{e}}_{(b)}-\tilde{A}^{(\hat{c})}_{(db)}\mathring{\tilde{f}}_{(c)}]\nonumber\\
&=H[\tilde{\nabla}_{(1)}\mathring{\tilde{e}}_{(b)}-\tilde{A}^{(\hat{c})}_{(1b)}\mathring{\tilde{f}}_{(c)}]
=H\tilde{\Gamma}_{(1b)}^{\;\;\;(d)}\mathring{\tilde{e}}_{(d)}
-H\tilde{A}_{(1b)}^{(\hat{d})}\mathring{\tilde{f}}_{(d)},\nonumber
\end{align}
where we have used the splitting of the ambient connection
\begin{equation}\label{e:connsplit1}
\overline{\nabla}_{\mathring{\tilde{e}}_{(a)}}\mathring{\tilde{e}}_{(b)}=\nabla_{\mathring{\tilde{e}}_{(a)}}\mathring{\tilde{e}}_{(b)}-A_{(ab)}^{(c)}\mathring{\tilde{f}}_{(c)}
\end{equation}
\begin{equation}\label{e:connsplit2}
\overline{\nabla}_{\mathring{\tilde{e}}_{(a)}}\mathring{\tilde{f}}_{(b)}=A^{(c)}_{(ab)}\mathring{\tilde{e}}_{(c)}+C_{(ab)}^{(c)}\mathring{\tilde{f}}_{(c)},
\end{equation}
where $C_{(ab)}^{(c)}$ are the components of the normal connection
\[
\nabla^\bot_{\mathring{\tilde{e}}_{(a)}}\mathring{\tilde{f}}_{(b)}=C_{(ab)}^{(c)}\mathring{\tilde{f}}_{(c)}.
\]
Similarly we find that
\[
\frac{\partial}{\partial s}\mathring{\tilde{f}}_{(b)}=H\tilde{A}_{(1\hat{b})}^{\;\;\;(d)}\mathring{\tilde{e}}_{(d)}
+H\tilde{C}_{(1\;\;\hat{b})}^{\;\;(\hat{d})}\mathring{\tilde{f}}_{(d)},
\]
\[
\frac{\partial}{\partial s}\mathring{\tilde{e}}^{(b)}=H\tilde{\Gamma}_{(1\;\;d)}^{\;\;(b)}\mathring{\tilde{e}}^{(d)}
-H\tilde{A}_{(1\hat{d})}^{(b)}\mathring{\tilde{f}}^{(d)},
\]
\[
\frac{\partial}{\partial s}\mathring{\tilde{f}}^{(b)}=H\tilde{A}_{(1d)}^{(\hat{b})}\mathring{\tilde{e}}^{(d)}
+H\tilde{C}_{(1\hat{d})}^{\;\;(\hat{b})}\mathring{\tilde{f}}^{(d)}.
\]
We now compute the evolution of the (co)frame of the flowing surface. To this end let
\[
\frac{\partial}{\partial s}\mathring{e}^{(a)}=M_{(b)}^{(a)}\mathring{e}^{(b)}+N_{(b)}^{(a)}\mathring{f}^{(b)},
\]
\[
\frac{\partial}{\partial s}\mathring{f}_{(a)}=K_{(a)}^{(b)}\mathring{e}_{(b)}+L_{(a)}^{(b)}\mathring{f}_{(b)},
\]
where $K,L,M$ and $N$ are $2\times 2$ matrices to be determined.

Differentiating $<\mathring{e}^{(a)},\mathring{f}_{(b)}>=0$ with respect to time implies that $N_{(b)}^{(a)}=-K_{(b)}^{(a)}$ and a standard calculation (see e.g. Lemma 2.2 of \cite{chenli}) shows that
\[
K_{(b)}^{(a)}=\mathring{e}^{(a)j}\nabla_jH_{(b)}+H^{(c)}C^{(a)}_{(\hat{c}\hat{b})},
\]
and so we have 
\[
\frac{\partial}{\partial s}\mathring{e}^{(a)}=M_{(b)}^{(a)}\mathring{e}^{(b)}
-\left[\nabla^{(a)}H_{(b)}+HC^{(a)}_{(\hat{1}\hat{b})}\right]\mathring{f}^{(b)},
\]
\[
\frac{\partial}{\partial s}\mathring{f}_{(a)}=\left[\nabla^{(b)}H_{(a)}+HC^{(b)}_{(\hat{1}\hat{a})}\right]\mathring{e}_{(b)}
+L_{(a)}^{(b)}\mathring{f}_{(b)},
\]
with $L$ and $M$ to be determined. This we do by the time derivative of the intersection and constant angle conditions in adapted frames:
\[
0=\frac{\partial}{\partial s}<\mathring{e}^{(a)},\mathring{\tilde{e}}_{(b)}>=\frac{\partial}{\partial s}<\mathring{e}^{(a)},\mathring{\tilde{f}}_{(b)}>
=\frac{\partial}{\partial s}<\mathring{f}_{(a)},\mathring{\tilde{e}}^{(b)}>=\frac{\partial}{\partial s}<\mathring{f}_{(a)},\mathring{\tilde{f}}^{(b)}>.
\] 
For example, consider
\begin{align}
0&=\frac{\partial}{\partial s}<\mathring{f}_{(1)},\mathring{\tilde{e}}^{(1)}>=<\frac{\partial}{\partial s}\mathring{f}_{(1)},\mathring{\tilde{e}}^{(1)}>
+<\mathring{f}_{(1)},\frac{\partial}{\partial s}\mathring{\tilde{e}}^{(1)}>\nonumber\\
&=<(\nabla^{(b)}H_{(1)}+HC^{(b)}_{(\hat{1}\hat{1})})\mathring{e}_{(b)}+L_{(1)}^{(b)}\mathring{f}_{(b)},\mathring{\tilde{e}}^{(1)}>
+H<\mathring{f}_{(1)},\tilde{\Gamma}_{(1\;\;d)}^{\;\;(1)}\mathring{\tilde{e}}^{(d)}
-\tilde{A}_{(1\hat{d})}^{(1)}\mathring{\tilde{f}}^{(d)}>\nonumber\\
&=-\nabla^{(1)}H+H\left[\sinh B\; \tilde{\Gamma}_{(1\;\;2)}^{\;\;(1)}-\cosh B\;\tilde{A}_{(1\hat{1})}^{(1)}\right]\nonumber\\
&=-\nabla_{(1)}H+H\left[\sinh B\;\tilde{\Gamma}_{(112)}+\cosh B\;\tilde{A}_{(11)}^{(\hat{1})}\right],\nonumber
\end{align}
which is the first of equations (\ref{e:dbc1}).
Similarly we find that
\[
0=L_{(1)}^{(1)}=L_{(2)}^{(2)}=M_{(1)}^{(1)}=M_{(2)}^{(2)},
\]
\[
L_{(1)}^{(2)}=-L_{(2)}^{(1)}=-H(\sinh B\tilde{A}^{(\hat{2})}_{(12)}-\cosh B\;\tilde{C}_{(1\hat{1}\hat{2})}),
\]
\[
M_{(1)}^{(2)}=-M_{(2)}^{(1)}=-H(\cosh B\;\tilde{\Gamma}_{(112)}+\sinh B\tilde{A}^{(\hat{1})}_{(11)}),
\]
and the rest of equations (\ref{e:dbc1}) and (\ref{e:dbc2}).

Equations (\ref{e:dbc3}) to (\ref{e:dbc6}) follow from the space derivatives along the intersection and constant angle conditions. For example taking the covariant derivative of this w.r.t. the ambient connection we find
\[
<\overline{\nabla}_{\mathring{e}_{(1)}}\mathring{e}_{(a)},\mathring{\tilde{e}}^{(b)}>
+<\mathring{e}_{(a)},\overline{\nabla}_{\mathring{e}_{(1)}}\mathring{\tilde{e}}^{(b)}>=0.
\]
From decomposing the ambient connection as in equation (\ref{e:connsplit1}), for $a=2$ and $b=1$ we get
\[
<\overline{\nabla}_{\mathring{e}_{(1)}}\mathring{e}_{(2)},\mathring{\tilde{e}}^{(1)}>=<\nabla_{\mathring{e}_{(1)}}\mathring{e}_{(2)}
-A_{(12)}^{(\hat{c})}\mathring{f}_{(c)},\mathring{e}^{(1)}>=\Gamma_{(12)}^{\;\; (1)},
\]
and equation (\ref{e:connsplit1}) along with the relationship between the two sets of frames also gives
\begin{align}
<\mathring{e}_{(2)},\nabla_{\mathring{e}_{(1)}}\mathring{\tilde{e}}^{(1)}>&=<\cosh B\;\mathring{\tilde{e}}_{(2)}
+\sinh B \;\mathring{\tilde{f}}_{(1)},\tilde{\nabla}_{\mathring{\tilde{e}}_{(1)}}\mathring{\tilde{e}}^{(1)}
    -\tilde{A}_{(11)}^{(\hat{c})}\mathring{\tilde{f}}_{(c)}>\nonumber\\
&=\cosh B\; \tilde{\Gamma}_{\;\;(11)}^{\;\; (2)}-\sinh B \;\tilde{A}_{(11)}^{(\hat{1})}.\nonumber
\end{align}
Putting these together we find
\[
\Gamma_{(12)}^{\;\; (1)}+\cosh B\; \tilde{\Gamma}_{(11)}^{\;\; (2)}-\sinh B \;\tilde{A}_{(11)}^{(\hat{1})}=0,
\]
which can be rearranged to give equation (\ref{e:dbc4}). The other equations follow similarly by covariantly differentiating the inner product of other frame vectors.
\end{proof}
\vspace{0.1in}

We now bound the second derivatives  of the flowing surface by those of the boundary plane along the edge.

\begin{Thm}\label{t:bdryest3}
Suppose that the initial conditions $B$ and $C$ are chosen to satisfy inequalities (\ref{e:smallness1}), (\ref{e:smallness2}) and (\ref{e:smallness3})
and $0<\theta<\pi/2$, $0<\psi<\pi/2$. 
Then for as long as the flow exists $|d^2F_s|<C_1(|d^2\tilde{F}|)$ at the edge.
\end{Thm}
\begin{proof}
The idea of the proof is to use equations (\ref{e:dbc2}) to (\ref{e:dbc7}), together with the spatial derivative of the asymptotic holomorphicity condition, to bound the derivatives of the 1-jet of $F$, as expressed by the quantities $\lambda$, $\mu$, $\phi$ and $\vartheta$. Note that by the identity (\ref{e:id1}) the derivatives of $\vartheta={\mathbb R}{\mbox{e}}(\rho)$ are determined by those of $\lambda$, $\mu=|\sigma|/|\lambda|$ and $\phi={\mbox {arg}}(\sigma)$ - an artifact of partial derivatives commuting.

The 2-jet of the flowing surface will ultimately be bounded by the 2-jet of the boundary plane. As we proceed we gather these edge 2-jet terms together on the right-hand side and denote them by a generic function $\tilde{\mathcal F}$ which is bounded for a $C^2$ boundary plane. That is, we write  $\tilde{\mathcal F}$ (possibly with a subscript) for quantities that only depend on the slopes $\tilde{\lambda}$, $\tilde{\mu}$, $\tilde{\phi}$ and $\tilde{\vartheta}$ and their first derivatives.

A complicating factor in this scheme is that derivatives of the angles $\theta,\tilde{\theta},\psi,\tilde{\psi}$ also enter into some computations and have to be controlled. 

The estimates will hold for any point $\gamma\in f_s(\partial D)$ and for computational simplicity we translate and rotate so that the point is $\xi=\eta=0$. The derivatives at such a point naturally split into normal and tangential components relative to the edge of the flowing disc and for any function $\Phi:D\rightarrow{\mathbb R}$ we denote the tangential and normal derivatives by
\[
\Phi_{|1}=\mathring{e}_{(1)}^j\frac{\partial \Phi}{\partial x^j},
\qquad \qquad
\Phi_{|2}=\mathring{e}_{(2)}^j\frac{\partial \Phi}{\partial x^j},
\]
respectively, where $\{\mathring{e}_{(1)},\mathring{e}_{(2)}\}$ are the adapted frames.

The expressions which link equations (\ref{e:dbc2}) to (\ref{e:dbc7}) to the derivatives of the 1-jet are the projections of the submanifold equations (\ref{e:connsplit1}) and (\ref{e:connsplit2}):
\begin{equation}\label{e:defsff}
A^{(\hat{c})}_{(ab)}=\mathring{f}^{(c)}_{j}\;^\perp P_{k}^j\mathring{e}_{(a)}^{\;l}\overline{\nabla}_l\mathring{e}_{(b)}^{\;k},
\end{equation}
\begin{equation}\label{e:defconns}
C^{\;\;(\hat{2})}_{(a)\;\;(\hat{1})}=\mathring{e}_{(a)}^{j}\;^\parallel P_j^k\mathring{f}^{(2)}_{l}\overline{\nabla}_k\mathring{f}_{(1)}^{\;l},
\qquad\qquad
\Gamma^{\;\;(2)}_{(a)\;\;(1)}=\mathring{e}_{(a)}^{j}\;^\parallel P_j^k\mathring{e}^{(2)}_{l}\overline{\nabla}_k\mathring{e}_{(1)}^{\;l},
\end{equation}
the expression for the aligned frames
\[
\mathring{e}_{(1)}=\cos\theta {e}_{(1)}-\sin\theta {e}_{(2)},
\qquad\qquad
\mathring{e}_{(2)}=\sin\theta {e}_{(1)}+\cos\theta {e}_{(2)},
\] 
\[
\mathring{f}_{(1)}=\cos\psi {f}_{(1)}+\sin\psi {f}_{(2)},
\qquad\qquad
\mathring{f}_{(2)}=-\sin\psi {f}_{(1)}+\cos\psi {f}_{(2)},
\] 
and the coordinate expressions in Proposition \ref{p:frames} for the canonical frames. Here the projection operators are exactly as in the proof of Proposition \ref{p:2ndff}. Similar expressions also hold for the boundary plane with tilded quantities.

We want to bound $\lambda_{|1}$, $\lambda_{|2}$, $\mu_{|1}$, $\mu_{|2}$, $\phi_{|1}$ and $\phi_{|2}$, as well as the gauge terms $\theta_{|1}$, $\tilde{\theta}_{|1}$, $\psi_{|1}$ and $\tilde{\psi}_{|1}$.

This last term can be bounded by noting that, by the first of equations (\ref{e:dbc2}) and (\ref{e:dbc7}) we have
\[
-\tilde{A}^{(\hat{2})}_{(11)}=C_{(\hat{1}\hat{2})}^{(1)}=C^{\;\;(\hat{2})}_{(1\hat{1})}
=-\sinh B\;\tilde{A}^{(\hat{2})}_{(12)}-\cosh B\; \tilde{C}^{\;\;(\hat{2})}_{(1\hat{1})}.
\]
Now since $\tilde{A}^{(\hat{2})}_{(11)}$ consists of 2-jet terms of the boundary and by the first expression in
(\ref{e:defconns}) we find that
\[
\tilde{C}^{\;\;(\hat{2})}_{(1\hat{1})}=\tilde{\psi}_{|1}+\tilde{\mathcal F}_1=\cosh B \tilde{A}^{(\hat{2})}_{(11)}=\tilde{\mathcal F}_2.
\]
Thus we have 
\[
\tilde{\psi}_{|1}=\tilde{\mathcal F}_3,
\]
and our first bound.

The tangential derivative of the asymptotic holomorphic condition (iv) of the {\bf I.B.V.P.} can be written, as per definition \ref{d:spinco}, 
\begin{equation}\label{e:dlam01}
\lambda_{|1}={\textstyle{\frac{\lambda}{\mu}}}\mu_{|1}.
\end{equation}
Moreover, we have from equations (\ref{e:dbc6}) and (\ref{e:dbc3}) that
\begin{align}
\sin{\scriptstyle{(\theta-\psi)}}A^{(\hat{1})}_{(12)} -\cos{\scriptstyle{(\theta-\psi)}}A^{(\hat{2})}_{(12)}
&=\sin{\scriptstyle{(\theta-\psi)}}\tilde{A}^{(\hat{1})}_{(12)}  -\cos{\scriptstyle{(\theta-\psi)}}
[\cosh B\;\tilde{A}^{(\hat{2})}_{(12)}+\sinh B\; \tilde{C}^{\;\;(\hat{2})}_{(1\hat{1})}]\nonumber\\
&=\tilde{\mathcal{F}}.\nonumber
\end{align}
Now using the explicit expression in Proposition \ref{p:2ndff} for the second fundamental form in terms of the 2-jet of the graph function we find that
\begin{equation}\label{e:dlam02}
\lambda_{|2}=\frac{\lambda}{1-\mu^2}\left[(1+\mu-2\sin^2\theta)\mu_{|2}
-2\cos\theta\sin\theta\mu(1-\mu^2)^{\scriptstyle{\frac{1}{2}}}\phi_{|2}\right]+\tilde{\mathcal F}.
\end{equation}
Equations (\ref{e:dlam01}) and (\ref{e:dlam02}) bound the derivatives of $\lambda$ by the other derivatives and so we can reduce our list of derivatives, leaving us with $\mu_{|1}$, $\phi_{|1}$, $\theta_{|1}$, $\tilde{\theta}_{|1}$, $\psi_{|1}$, $\mu_{|2}$ and  $\phi_{|2}$. These split naturally into tangent and normal derivatives which we now exploit.

The five remaining equations of (\ref{e:dbc2}) to (\ref{e:dbc7}) give a linear system for the five tangential derivatives which is of the form 
\[
\left[\begin{matrix}
*&* &0 &0 &0 \\
*&* &0 &0 &* \\
*&* &0 &0 &* \\
*&* &0 &1 &0 \\
*&* &-1 &0 &* 
\end{matrix}
\right]
\left[\begin{matrix}
{\textstyle{\frac{1}{2\mu(1-\mu^2)}}}\mu_{|1}\\
{\textstyle{\frac{\mu}{2\sqrt{1-\mu^2}}}}\phi_{|1}\\
\theta_{|1}\\
\psi_{|1}\\
\sinh B\tilde{\theta}_{|1}
\end{matrix}
\right]=\tilde{\mathcal F}.
\]
Note that $|d\sigma|^2=d|\sigma|^2+|\sigma|^2|d\phi|^2$ and so to bound the gradients of $\sigma$ we need only bound $d|\sigma|$ and $|\sigma|d\phi=|\lambda|\mu d\phi$.

The last two equations imply that bounds on $\mu_{|1}$, $\phi_{|1}$ and $\tilde{\theta}_{|1}$ yield bounds on $\theta_{|1}$ and $\psi_{|1}$. Thus we are left with a $3\times 3$ system for $\mu_{|1}$, $\phi_{|1}$ and $\tilde{\theta}_{|1}$ with matrix of coefficients
\[
\left[\begin{matrix}
1-\cos 2\theta\mu&\sin 2\theta &0 \\
 & & \\
\sin 2\psi +(\sin 2\theta \cos 2\psi+\sin 2(\theta+\psi))\mu&\cos 2\theta\cos 2\psi+\cos 2(\theta+\psi)&\cos(\theta+\psi) \\
 & & \\
\cos 2\psi-(\sin 2\theta\sin 2\psi -\cos 2(\theta+\psi))\mu&- \cos 2\theta\sin 2\psi -\sin 2(\theta+\psi) &-\sin(\theta+\psi)
\end{matrix}
\right].
\]
The determinant of this matrix is 
\[
\cos\psi\sin\theta(1+\mu)+\cos\theta\sin\psi(1-\mu)>0,
\]
where the inequality follows from (\ref{e:smallness1}), (\ref{e:smallness2}) and (\ref{e:smallness3}). We can therefore invert the system to yield bounds on the tangential derivatives: $\mu_{|1}$, $\phi_{|1}$ and $\tilde{\theta}_{|1}$. The final tangential derivative we have to bound is $\vartheta_{|1}$ and this follows readily from differentiating equation (\ref{e:hypangcon3}) along the edge.

We turn now to bounding the normal derivatives $\mu_{|2}$ and $\phi_{|2}$. These are controlled by the tangential derivatives due to the generalized Coddazzi-Mainardi equation (\ref{e:id1}). To see this, introduce
\[
\alpha_\theta=\cos\theta\alpha_1-\sin\theta\alpha_2
={\textstyle{\frac{1}{\sqrt{2}}}}\left(\frac{\cos\theta}{(|\lambda|-|\sigma|)^{\scriptstyle{\frac{1}{2}}}}
+\frac{i\sin\theta}{(|\lambda|+|\sigma|)^{\scriptstyle{\frac{1}{2}}}}\right)e^{-{\scriptstyle{\frac{i}{2}}}\phi+{\scriptstyle{\frac{i}{4}}}\pi},
\]
so that the tangential and normal derivative operators are
\[
\mathring{e}_{(1)}=\alpha_\theta\frac{\partial}{\partial \xi}+\overline{\alpha}_\theta\frac{\partial}{\partial \bar{\xi}},
\qquad\qquad
\mathring{e}_{(2)}=\alpha_{\theta-{\scriptstyle{\frac{\pi}{2}}}}\frac{\partial}{\partial \xi}
+\overline{\alpha}_{\theta-{\scriptstyle{\frac{\pi}{2}}}}\frac{\partial}{\partial \bar{\xi}},
\]
which can be inverted to
\[
\frac{\partial}{\partial \xi}=
-i(\lambda^2-|\sigma|^2)^{\scriptstyle{\frac{1}{2}}}\left[\overline{\alpha}_{\theta-{\scriptstyle{\frac{\pi}{2}}}}\mathring{e}_{(1)}
-\overline{\alpha}_\theta\mathring{e}_{(2)}\right].
\]

Equation (\ref{e:id1}) says that
\[
\partial\bar{\sigma}-\bar{\partial}\vartheta-i\bar{\partial}\lambda=\tilde{\mathcal F},
\]
which relates certain tangential and normal derivatives. Thus
\[
\vartheta_{1}=\alpha_\theta\partial\vartheta+\overline{\alpha}_\theta\bar{\partial}\vartheta
=\alpha_\theta(i\partial\lambda+\bar{\partial}\sigma)+\overline{\alpha}_\theta\bar{\partial}
(-i\bar{\partial}\lambda+\partial\bar{\sigma})+\tilde{\mathcal F}.
\]
Now a direct computation reduces this to
\[
-(1-\cos\theta)(\mu+\cos 2\theta)\mu_{|2}+\sin 2\theta(1-\cos\theta)\mu(1-\mu^2)^{\scriptstyle{\frac{1}{2}}}\phi_{|2}=\tilde{\mathcal F}.
\]
On the other hand equation (\ref{e:compat}) implies that
\[
\cos(\theta+\psi)\mu_{|2}-\sin(\theta+\psi)\mu(1-\mu^2)^{\scriptstyle{\frac{1}{2}}}\phi_{|2}=\tilde{\mathcal F}.
\]
The last two equations form a linear system for $\mu_{|2}$ and $\phi_{|2}$ whose coefficients have determinant
\[
-(1-\cos\theta)\mu(1-\mu^2)^{\scriptstyle{\frac{1}{2}}}[\cos\psi\cos\theta(1+\mu)+\sin\psi\sin\theta(1-\mu)].
\]
Inequalities (\ref{e:smallness1}), (\ref{e:smallness2}) and (\ref{e:smallness3}) ensure that this remains strictly negative during the flow and hence we can invert the system and get bounds for $\mu_{|2}$ and $\phi_{|2}$. This completes the proof.
\end{proof}

\vspace{0.1in}

\subsection{{Existence of a holomorphic disc}}\label{ss:hol_disc}

In this section we prove that we can attach a holomorphic disc to the set of oriented normals of any non-umbilic convex hemisphere in ${\mathbb R}^3$, considered as a surface in $TS^2$.

We first prove a final estimate.

\begin{Prop}\label{p:shearest}
Under the mean curvature flow we have the following estimate:
\[
\left(\frac{\partial }{\partial s}-{\mathbb G}^{jk}\partial_j\partial_k\right)\left(
\frac{|\sigma|^2}{\lambda^2-|\sigma|^2}\right)\leq
    \frac{4\lambda}{|\sigma|^2}\frac{|\sigma|^4}{(\lambda^2-|\sigma|^2)^2},
\]
\end{Prop}
\begin{proof}
Our starting point is equation (\ref{e:holconv}), which we rewrite in the form
\[
\left(\frac{\partial }{\partial s}-{\mathbb G}^{jk}\partial_j\partial_k\right)\left(
\frac{|\sigma|^2}{\lambda^2-|\sigma|^2}\right)=I_1+I_2+I_3+I_4,
\]
where
\[
I_1=-2\frac{(\lambda^2+|\sigma|^2)}{(\lambda^2-|\sigma|^2)^3}
    \Big\|\lambda d|\sigma|-|\sigma|d\lambda -\frac{\lambda^2-|\sigma|^2}{\lambda^2+|\sigma|^2}\frac{\lambda|\sigma|}{1+\xi\bar{\xi}}d(1+\xi\bar{\xi})\Big\|^2,
\]
\[
I_2=2\frac{(1+\xi\bar{\xi})\lambda^2|\sigma|}{(\lambda^2-|\sigma|^2)^3}<d(1+\xi\bar{\xi}),\lambda d|\sigma|-|\sigma|d\lambda>,
\]
\[
I_3=-2\frac{\lambda^2|\sigma|^2}{(\lambda^2-|\sigma|^2)^3}\Big\|d\phi-2(1+\xi\bar{\xi})^{-1}j[d(1+\xi\bar{\xi})]\Big\|^2,
\]
\[
I_4=\frac{|\sigma|^2}{2(\lambda^2-|\sigma|^2)^2(\lambda^2+|\sigma|^2)}\Big\{-i|\sigma|\lambda^2
    (\xi^2e^{i\phi}-\bar{\xi}^2e^{-i\phi})+4\lambda\{[2-\xi\bar{\xi}]\lambda^2+2|\sigma|^2\}\Big\}.
\]
Here $\sigma=|\sigma|e^{i\phi}$ and we have introduced the complex structure $j(d\xi)=id\xi$.

Introduce the flat metric
\[
<d\xi,d\bar{\xi}>=1, \qquad \qquad <d\xi,d\xi>=<d\bar{\xi},d\bar{\xi}>=0,
\]
on ${\mathcal P}$ via its coordinate $\xi$. Denote the flat norm and inner product by $|.|$ and $<\cdot,\cdot>$, and the norm and inner product of $g$ by $\|.\|$ and $<<\cdot,\cdot>>$. The following estimates will prove useful:

\begin{Lem}\label{l:metrics}
\[
\frac{(1+\xi\bar{\xi})^2(-\lambda-|\sigma|)}{\lambda^2-|\sigma|^2}|X|^2\leq\|X\|^2\leq\frac{(1+\xi\bar{\xi})^2(-\lambda+|\sigma|)}{\lambda^2-|\sigma|^2}|X|^2.
\]
\end{Lem}

\vspace{0.1in}

First we estimate $I_1$ using the flat metric and Lemma \ref{l:metrics}:
\[
I_1\leq-2\frac{(1+\xi\bar{\xi})^2(\lambda^2+|\sigma|^2)}{(\lambda^2-|\sigma|^2)^3(-\lambda+|\sigma|)}
    \Big|\lambda d|\sigma|-|\sigma|d\lambda -\frac{\lambda^2-|\sigma|^2}{\lambda^2+|\sigma|^2}\frac{\lambda|\sigma|}{1+\xi\bar{\xi}}
   d(1+\xi\bar{\xi})\Big|^2,
\]
and so after completing the squares
\begin{align}
I_1+I_2&\leq-2{\textstyle{\frac{(1+\xi\bar{\xi})^{2}(\lambda^2+|\sigma|^2)}{(\lambda^2-|\sigma|^2)^3(-\lambda+|\sigma|)}}}\Bigg\{
    \Big|\lambda d|\sigma|-|\sigma|d\lambda +{\textstyle{\frac{(-\lambda+|\sigma|)(-\lambda+2|\sigma|)\lambda|\sigma|}{2(\lambda^2+|\sigma|^2)(1+\xi\bar{\xi})}}}d(1+\xi\bar{\xi})\Big|^2\nonumber\\
&\qquad\qquad\qquad\qquad-{\textstyle{\frac{\lambda(-\lambda+|\sigma|)^2(-3\lambda-4|\sigma|)}{4(\lambda^2+|\sigma|^2)^2}\frac{\lambda^2|\sigma|^2}{(1+\xi\bar{\xi})^2}}}
    \Big|d(1+\xi\bar{\xi})\Big|^2\Bigg\}\nonumber.
\end{align}

Clearly $I_3$ is negative, so we discard it. To estimate $I_4$ we use
\[
-2\xi\bar{\xi}\leq i(\xi^2e^{i\phi}-\bar{\xi}^2e^{-i\phi})\leq 2\xi\bar{\xi}.
\]
Thus
\[
I_4\leq{\textstyle{\frac{|\sigma|^2}{2(\lambda^2-|\sigma|^2)^2(\lambda^2+|\sigma|^2)}}}
    \Big\{2|\sigma|\lambda^2\xi\bar{\xi}+2\lambda\{[4-2\xi\bar{\xi}]\lambda^2+4|\sigma|^2\}\Big\}.
\]
Combining the estimates
\[
\left(\frac{\partial }{\partial s}-{\mathbb G}^{jk}\partial_j\partial_k\right)\left(\frac{|\sigma|^2}{\lambda^2-|\sigma|^2}\right)\leq\frac{4\lambda|\sigma|^2}{(\lambda^2-|\sigma|^2)^2}
    +\frac{\lambda^2|\sigma|^2(\lambda^3+2\lambda^2|\sigma|-2\lambda|\sigma|^2-|\sigma|^3)\xi\bar{\xi}}{(\lambda^2-|\sigma|^2)^3(\lambda^2+|\sigma|^2)}.
\]
In fact, we can achieve $|\lambda|\geq3|\sigma|$ throughout the flow (see Corollary \ref{c:bdryests} with $C_3<1/3$), so that
\begin{align}
\left(\frac{\partial }{\partial s}-{\mathbb G}^{jk}\partial_j\partial_k\right)\left(\frac{|\sigma|^2}{\lambda^2-|\sigma|^2}\right)&\leq\frac{4\lambda|\sigma|^2}{(\lambda^2-|\sigma|^2)^2}
    +\frac{\lambda^5|\sigma|^2\xi\bar{\xi}}{9(\lambda^2-|\sigma|^2)^3(\lambda^2+|\sigma|^2)}\nonumber\\
& \leq\frac{4\lambda}{|\sigma|^2}\frac{|\sigma|^4}{(\lambda^2-|\sigma|^2)^2}\nonumber.
\end{align}
This completes the proof of the Proposition.
\end{proof}

\vspace{0.1in}

We now show that our flow is asymptotically holomorphic:

\begin{Prop}\label{p:asymhol}
The mean curvature flow satisfies:
\[
|\sigma|^2\rightarrow 0 \qquad\qquad{\mbox as} \qquad s\rightarrow\infty.
\]
\end{Prop}
\begin{proof}
At the edge this follows from the second Neumann condition. To consider the interior, recall the following estimate:
\[
\left(\frac{\partial }{\partial s}-{\mathbb G}^{jk}\partial_j\partial_k\right)\left(
\frac{|\sigma|^2}{\lambda^2-|\sigma|^2}\right)\leq
    \frac{4\lambda}{|\sigma|^2}\frac{|\sigma|^4}{(\lambda^2-|\sigma|^2)^2}.
\]
This is of the form
\[
\left(\frac{\partial }{\partial s}-{\mathbb G}^{jk}\partial_j\partial_k\right)f\leq-C_1^2f^2,
\]
for the positive function $f$ given by 
\[
f=\frac{|\sigma|^2}{\lambda^2-|\sigma|^2}, 
\]
and $C_1$ is a constant such that
\[
\frac{4|\lambda|}{|\sigma|^2}\geq C_1^2,
\]
which exists by Corollary \ref{c:bdryests}. Following Ecker and Huisken \cite{EaH}, choose $g=sf$ and compute
\[
\left(\frac{\partial }{\partial s}-{\mathbb G}^{jk}\partial_j\partial_k\right)g\leq gs^{-1}(1-C_1^2g).
\]
Now suppose that the maximum of $g$ occurs in the interior of the disc. Then, by the maximum principle we must have $1-C_1^2g\geq0$ or, equivalently, $f\leq C_1^{-2}s^{-1}$. Returning to our notation, this means that
\[
\frac{|\sigma|^2}{\lambda^2-|\sigma|^2} \leq \frac{1}{C_1^2s},
\]
or 
\[
|\sigma|^2\leq\frac{\lambda^2-|\sigma|^2}{C_1^2s} \leq \frac{C_2}{s}.
\]
That is, the flow is asymptotically holomorphic.
\end{proof}

\vspace{0.1in}

Finally, drawing the results together, the existence of the holomorphic disc is established as follows.

\vspace{0.1in}

\begin{Thm}\label{t:ashola}
Let $S$ be a $C^{3+\alpha}$ smooth open convex plane in ${\mathbb R}^3$ without umbilic points and suppose that the Gauss image of $S$ contains a closed hemisphere. Let ${\mathcal P}\subset TS^2$ be the oriented normals of $S$. 

Then $\exists f:D\rightarrow TS^2$ with $f\in C^{1+\alpha}_{loc}(D)\cap C^0(\overline{D})$ satisfying
\begin{enumerate}
\item[(i)] $f$ is holomorphic,
\item[(ii)] $f(\partial D)\subset {\mathcal P}$.
\end{enumerate}
\end{Thm} 
\begin{proof}
By a rotation in ${\mathbb R}^3$ and the induced action on $TS^2$  we can take the north pole of ${\mathcal P}$ to $\xi=0$. Now deform ${\mathcal P}$ to $\tilde{\mathcal P}$ by adding a holomorphic twist. Thus if ${\mathcal P}$ is given by the graph function $\eta=F(\xi,\bar{\xi})$, then $\tilde{\mathcal P}$ is given by the graph function $\eta=\tilde{F}(\xi,\bar{\xi})=F(\xi,\bar{\xi})-iC_0\xi$. We choose $C_0>0$ large enough so that $\tilde{\mathcal P}$ is positive at the pole.

We can now apply Theorem \ref{t:lte} to find a long-time solution $f\in C^{2+\alpha}_{loc}(D\times{\mathbb R}_{\geq0})\cap C^1(\overline{D}\times{\mathbb R}_{\geq0})$ to mean curvature flow with edge in $\tilde{\Lambda}$, so long as the initial conditions $B$ and $C$ are chosen small enough. Moreover, in Proposition \ref{p:asymhol} we showed that this solution is asymptotically holomorphic in time:
\[
|\sigma|^2\rightarrow 0 \qquad\qquad{\mbox as} \qquad s\rightarrow\infty.
\]
Now by parabolic Schauder estimates, (see e.g. \cite{LaS} section 6, and \cite{JJ} page 79), we have 
\[
||f(\cdot,t)||_{C_{loc}^{1+\alpha}(D)} \le 
C(||H||_{L^\infty(D \times [0,\infty))} + ||\overline{K}(\cdot,t)||_{L^\infty(D)}).
\]
Here, $\overline{K}$ involves ambient metric, Christoffel symbols, and the gradient $f$ (all of which are bounded since the evolution takes place in a relatively compact subset of $TS^2$), and on the mean curvature vector $H$, which is bounded for all time. The right hand side being bounded in time, and using the gradient bound from uniform positivity, we can by Arzela-Ascoli extract a subsequence $t_j \to \infty$ limit disc $\tilde{f}_\infty(D)$, where $\tilde{f}_\infty \in C_{loc}^{1+\alpha'}(D)\cap C^0(\overline{D})$, for $\alpha' < \alpha$. 

From asymptotic holomorphicity, proved in the next section, it now follows that $\tilde{f}_\infty(D)$ is holomorphic with respect to ${\mathbb J}$. Note that we do {\it not} have smooth convergence up to the boundary, and that (in general), the angle condition (iii) in I.B.V.P. is not retained by the holomorphic limit $\tilde{f}_\infty(D)$.

Finally, the holomorphic disc $f_\infty(D)$ with edge lying on ${\mathcal P}$ can now be obtained by subtracting the holomorphic twist.

\end{proof}

\vspace{0.1in}
\section{Proof of Theorem 1}\label{s:proof}

The proof of Theorem 1 now goes as follows. Suppose that a plane $P_0\subset{\mathbb R}^3$ exists that satisfies Properties I, II and III, and let ${\mathcal P}_0\subset{\mathbb L}$ be the set of normals to $P_0$. Then by Proposition \ref{p:banach}, ${\mathcal P}_0$ lies in an open subset of  $C^{2+\alpha}$-smooth Lagrangian sections. This open set is a Banach manifold modeled on the Lagrangian vector fields which arise as ${\mathbb J}$ times the tangent vectors of ${\mathcal P}_0$.

Moreover, Proposition \ref{p:banach2} establishes that, passing to a smaller neighbourhood, the space of holomorphic discs with edge lying on points of the neighbourhood is a submanifold of the space of all smooth
maps.

Since the Cauchy-Riemann operator with this boundary condition is Fredholm, applying the Sard-Smale theorem in Proposition \ref{p:fredholm} proves that for a dense open subset the Cauchy-Riemann operator with this boundary condition is surjective. We conclude that the dimension of the space of parameterized holomorphic discs is equal to the analytic index of the operator. By Proposition \ref{p:oh} this index is related to the Keller-Maslov index of the edge curve by $I = \mu + 2$.

Consider a holomorphic disc with edge lying on ${\mathcal P}_0$. By Proposition 2.1.3 of \cite{Gk01}, the Keller-Maslov index of the edge is the relative first Chern class of the edge which, for graphs over the same domain, counts the number of complex points inside the edge on ${\mathcal P}_0$.

Suppose that the disc in ${\mathcal P}_0$ bounded by the edge of the holomorphic disc is totally real, so that $\mu = 0$. Thus $I = \mu + 2 = 2$ and, quotienting by the M\"obius group of the disc, the space of unparameterised holomorphic discs is $I-3 = 2-3 = -1$. This means that, for $C^{2+\alpha}$ boundary conditions, there cannot exist a holomorphic disc with edge on ${\mathcal P}_0$.

We conclude that, should ${\mathcal P}_0$ exist, then so would a nearby plane over which it is {\em not} possible to attach a holomorphic disc. Finally by Theorem \ref{t:mcf}, we have seen that it is {\em always} possible to attach a holomorphic disc to any such plane. We therefore conclude that there does not exist a proper $C^{3+\alpha}$ plane $P_0\subset{\mathbb R}^3$ satisfying Properties I, II and III, thereby proving Theorem 1.

\end{document}